\documentclass[a4paper,12pt]{article}

\usepackage{amssymb,graphicx,multicol,footmisc,url,theorem}

\newcommand{\Ii}{1\!\!1}

\newcommand{\text}[1]{\mbox{#1}}

\newtheorem{theorem}{Theorem}[section]
\newtheorem{corollary}[theorem]{Corollary}
\newtheorem{lemma}[theorem]{Lemma}
\newtheorem{proposition}[theorem]{Proposition}
\newtheorem{definition}[theorem]{Definition}
{\theorembodyfont{\rm} \newtheorem{example}[theorem]{Example}
\newtheorem{remark}[theorem]{Remark} }
\newenvironment{proof}{\noindent\textbf{Proof.} }{\hfill $\blacksquare$}

\def\a0{a^{(0)}}
\newcommand{\email}[1]{\small \sl#1}
\newcommand{\institute}[1]{#1}
\newcommand{\keywords}[1]{\noindent \textbf{Keywords: }#1}
\newcommand{\bbbr}{\mathbb{R}}
\newcommand{\bbbn}{\mathbb{N}}
\newcommand{\bbbz}{\mathbb{Z}}

\newcommand{\bbbc}{\mathbb{C}}

\begin{document}

\title{The free Kawasaki dynamics of continuous particle systems in infinite volume}
\date{}

\author{\\Yuri G.~Kondratiev\\
\institute{\small Fakult{\"a}t f{\"u}r Mathematik, Universit{\"a}t Bielefeld, \small D 33501 Bielefeld, Germany
\\ \small Forschungszentrum BiBoS, Universit{\"a}t Bielefeld,
\small D 33501 Bielefeld, Germany}\\
\email{kondrat@mathematik.uni-bielefeld.de} \\[0.2cm]
Tobias Kuna\\ \institute{\small Dipartimento di Ingegneria, Scienze \small dell'Informazione e Matematica,\\ \small Universit\'a degli Studi dell'Aquila, \small I 67100 L'Aquila, Italy\\
\email{tobias.kuna@univaq.it}}\\[0.3cm]
Maria Jo{\~a}o Oliveira\\
\institute{\small DCeT, Universidade Aberta, \small P 1269-001 Lisbon, Portugal\\
\small CMAFCIO, University of Lisbon, \small 1749-016 Lisbon}\\
\email{mjoliveira@ciencias.ulisboa.pt}\\[0.3cm] 
Jos{\'e} Lu{\'\i}s da~Silva\\
\institute{\small CCM, University of Madeira, \small P 9000-390 Funchal, Portugal \\
\email{joses@staff.uma.pt}}\\[0.3cm] 
Ludwig Streit\\
\institute{\small Forschungszentrum BiBoS, Universit{\"a}t Bielefeld,
\small D 33501 Bielefeld, Germany \\
\email{streit@uma.pt}}
}

\maketitle

\newpage
\textit{In memory of our colleague and friend Maria Concei{\c c}{\~a}o Carvalho}

\begin{abstract}
An infinite particle system of independent jumping particles in infinite volume is considered. Their construction is recalled,
further properties are derived, the relation with hierarchical equations, Poissonian analysis, and second quantization are discussed. The hydrodynamic limit for a general initial distribution satisfying a mixing condition is derived. The long time asymptotic is computed under an extra assumption. The relation with constructions based on infinite volume limits is discussed.
\end{abstract}




\noindent {\bf 2000 AMS Classification:}   82C21, 60G55, 60J75, 37A60

\tableofcontents

\section{Introduction}

Here we study systems with a large number of particles where the dynamics consists of particles jumping at random times.  To gain insights to these systems one studies them at different scales which are relevant to the system, for us:  the scale on which the individual movement is happening, the typical distance between particles, the scale on which we observe the system and finally the size of the system. The key feature of this paper is that we consider the regime where the system size is much larger than all the other scales and try to evaluate what kind of challenges that may bring. The more commonly studied case is where the scale on which the system is observed and the system size are of the same order, which simplifies the situation substantially. Most considerations in this paper become almost trivial in this case, see Subsection~\ref{subsecfv} for a detailed discussion.

We will walk in this paper through mathematical and some conceptual challenges which appear when the system size is much larger than the other scales. For simplicity we model this through an infinite system, as it means all results derived are automatically independent of the system size. For example, relaxation rates, which depend on the system size, may easily give rise to so slow relaxation of large systems failing to capture the physically relevant relaxation mechanism. 
Infinite systems are very challenging, see  \cite{S89} and \cite{CMS05} for Hamiltonian dynamics and \cite{Fr87}, \cite{Sp86} for gradient diffusions. Despite serious efforts in both cases our understanding is far from complete and rigorous results are sparse.  In order to get a better understanding of the challenges one has to face treating infinite systems, we will consider a non-interacting system where each particle follows a jump process. This allows us to have a very explicit representation of the dynamics and to treat the following three properties in great detail: construction of dynamics (Section~\ref{Section2} and \ref{seqeq}), long time asymptotic (Subsection~\ref{subsecLargetime} and \ref{SubsecGenTemp}), and hydrodynamic limits (Subsection~\ref{SubSecHydro} and \ref{SubsecGenHydro}).

The study of the dynamics of the infinite particle system without interaction can be reduced to the analysis of a type of one-particle dynamics. However, this one-particle dynamics being an effective description of the infinite particle system dynamics has to be considered with initial "distributions" which are not integrable, but merely bounded, as they actually describe the particle density of the infinite particle system. To stress this different interpretation, we will call it a pseudo-one-particle system in the following. Each constant initial "distribution" gives rise to a distinct invariant measure of the infinite particle system, in our case, Poisson point random fields with constant intensity. These are all mutually singular probability measures and hence the system cannot be ergodic. If we want to understand the dynamics of the infinite particle system in full we are forced to work with initial distributions which are singular towards each other. In this paper, we consider increasingly more general initial distributions corresponding to different physical realms. In Section~\ref{seqeq} we consider system in equilibrium that is point random fields with density with respect to a Poisson point random field with constant intensity, in Section~\ref{led} locally equilibrium, that is Poisson point random fields with non-constant intensity and in Section~\ref{nonled} non-equilibrium, that is general point random fields with some moment conditions.

The reader may wonder whether in the free case not everything is totally obvious. We try to tackle all difficulties using the most straightforward explicit techniques to make the challenges transparent. All these difficulties are innate for large system size, see Subsection~\ref{subsecfv}. Though most of the techniques will not extend to the interacting case the challenges should be maintained if not increased.  

Let us give a more technical introduction into the paper.
In Section~\ref{SecDefKa}  we describe the generator of the infinite particle dynamics, the state space and associated notions. The dynamic considered here is a version in the continuum of the well established Kawasaki dynamics for lattice gas systems. These are random evolutions of particle systems in which  individual particles jump in the space with
rates leading to a Gibbs state in the continuum as an invariant
measure \cite{G82}, \cite{KLRII05} and \cite{BKO13}. Having the particle number as a conserved quantity these dynamics give rise to a continuous family of invariant measures which are important to study the so-called hydrodynamic limits, see e.g. \cite{PR91} and \cite{KL99}.
Continuous versions of Glauber dynamics were introduced in
\cite{G81}, \cite{BCC02}, \cite{KL05}, \cite{KLRII05}, \cite{GK06}, \cite{P08}, \cite{KKO13}.
In Subsection~\ref{SecQuasi} we introduce and recall properties of the associated (pseudo)-one-particle operator. 

In Section~\ref{Section2} we discuss different ways to construct the infinite particle dynamics. In Subsection~\ref{Subsection2.1} we recall a probabilistic approach worked out in \cite{KLR05}. The free infinite particle dynamics is just the product of countably many copies of the one particle dynamics. However, to consider this product as a stochastic movement in the configuration space as further step one has to show that a.s. in any finite time interval
only finitely many jump events are visible in a bounded observation
window, i.e., the number of particles in the window stays
finite and only finitely many particles pass the window. This does not hold for any initial configuration and the processes can only be started in a restricted set of initial configurations or in initial distributions supported therein. 
We add to the results of \cite{KLR05}
the description of the path-space measure corresponding to the process.

This construction method cannot be extended to interacting systems. A probabilistic construction writing the system as a limit of finite particle systems will require a quite detailed control, see e.g. \cite{Fr87}.  In Subsection~\ref{Subsection2.2} we present an analytic approach to construct the semi-group which may also work for interacting systems. The approach is based on writing the dynamics in terms of the system of correlation functions, see \cite{BKO13} for such a construction in the interacting case. If one only starts in probability measure with a density with respect to an invariant measure, that is in the equilibrium case, techniques from second quantisation become available, cf.~Section~\ref{seqeq}. In Section~\ref{SecFock} we give a description of the dynamics in terms of Fock spaces which allow also a comparison with non-relativistic quantum fields type descriptions. In the case of symmetric dynamics the 
powerful methods of Dirichlet forms may be applied, in particular, existence
can be shown even for general interacting systems,
cf.~\cite{KLRII05} and references therein. In the
case of the free Kawasaki process, one can
apply in addition second quantization techniques,
 which give a full description of the $L^2$-theory, also for
non-symmetric jump rates, cf.~Section~\ref{Aug08-eq2}.

In Section~\ref{led} we consider the case in which one starts the process in a Poisson point random field with non-constant intensity, that is the system is in local equilibrium. The path space measure is then a Poisson point random field on path space. $L^p$ techniques break down already in this case and hence Dirichlet form techniques cannot be applied, see Remark~\ref{remnosemi}. In Subsection~\ref{subsecLargetime} we consider the long time asymptotic of the one-dimensional distribution in time and show that it converges to an invariant measure, that is a Poisson random field with constant intensity. The constant is the arithmetic mean of the initial intensity, see Definition~\ref{defamean}. This is easily shown in Proposition~\ref{PrAsya} under the assumption that the Fourier transform of the intensity $z$ is a signed measure. For such $z$ the arithmetic mean exists, see Corollary~\ref{Colmeana}, but not any bounded non-negative function $z$ for which the arithmetic mean exists and the time asymptotic converges has a signed-measure as Fourier transform, see Remark~\ref{rem::fourieractiv}. There also exist  bounded non-negative functions for which the arithmetic mean does not exist, see Remark~\ref{Aug08-eq1}. In order to establish a general result one needs a better analytic control of the behaviour of $z$ at infinity.


In Subsection~\ref{SubSecHydro} we scale space and time simultaneously to obtain a macroscopic description in terms of partial differential equation. When the Fourier transform of the activity is a signed measure the proper choice of a scale and the computation of the limit is direct. For general bounded activities one has to control higher derivatives, see Proposition~\ref{Prohydrogen}.

In Subsection~\ref{subsecfv}, we compare the results for systems where the system size is of the same order as the scale on which one observes the system and we derive a representation of the continuous Fourier transform in order to see the technical difference with the previous results. The finite system size induces a spectral gap of size $L^{-2}$ which eliminates all regularity requirements on the jump rate and the activity $z$.  We also show that for mildly larger system size this effect essentially maintains.  In \cite{DSS82} the authors showed the convergence to the hydrodynamic limit, with much more sophisticated techniques, uniformly in time for an infinite particle system but of  independent Brownian particles. We also discuss why in the infinite volume one cannot give a comparison result between the free Kawasaki dynamics and the system of independent Brownian particles.

The above described results about long time asymptotic and
hydrodynamic limits are extended to general initial distributions
 in Section~\ref{nonled}. We require some kind of mixing property
 for the initial distributions,
the so-called decay of correlation, which we formulate in terms of
the cumulants (or Ursell functions). For the hydrodynamic limit we
have to require in addition that the first correlation function of
the initial measure converges reasonably under the scaling. 

We prove in Section~\ref{SecGibbs} that the conditions on the initial distributions required above
are fulfilled, in particular, by Gibbs measures in the high
temperature low activity regime.

\section{Kawasaki dynamics\label{SecDefKa}}

\subsection{The generator}

The configuration space $\Gamma :=\Gamma _{\bbbr^d}$ over $\bbbr^d$,
$d\in \bbbn$, is defined as the set of all locally finite subsets of
$\bbbr^d$,
\[
\Gamma :=\left\{ \gamma \subset \bbbr^d:\left| \gamma_\Lambda\right|
<\infty \hbox{
for every compact }\Lambda\subset \bbbr^d\right\} ,
\]
where $\left| \cdot \right| $ denotes the cardinality of a set and
$\gamma_\Lambda:= \gamma\cap\Lambda$. As usual we
identify each $\gamma \in \Gamma $ with the non-negative Radon measure
$\sum_{x\in \gamma }\delta_x\in \mathcal{M}(\bbbr^d)$, where
$\delta_x$ is the Dirac measure with mass at $x$,
$\sum_{x\in\emptyset}\delta_x$ is, by definition, the zero measure, and
$\mathcal{M}(\bbbr^d)$ denotes the space of all non-negative Radon
measures on the Borel $\sigma$-algebra $\mathcal{B}(\bbbr^d)$.
This identification allows to endow $\Gamma $ with the topology induced by the
vague topology on $\mathcal{M}(\bbbr^d)$, i.e., the weakest topology on
$\Gamma$ with respect to which all mappings
\[
\Gamma \ni \gamma \longmapsto \langle f,\gamma\rangle :=
\int_{\bbbr^d}\gamma(dx)\,f(x)=\sum_{x\in \gamma }f(x),\quad
f\in C_c(\bbbr^d),
\]
are continuous. Here $C_c(\bbbr^d)$ denotes the set of all continuous
functions on $\bbbr^d$ with compact support. By $\mathcal{B}(\Gamma )$ we
will denote the corresponding Borel $\sigma$-algebra on $\Gamma$.

Given a non negative function $a \in L^1(\bbbr^d,dx)$, the
generator $L:=L_a$ of the free Kawasaki dynamics for an infinite
particle system is given by the informal expression
\begin{equation}\label{DefKawaOp}
(LF)(\gamma) := \sum_{x \in \gamma} \int_{\bbbr^d} dy\, a(x-y)
\left( F(\gamma \setminus x \cup y)- F(\gamma) \right).
\end{equation}
We proceed to give a rigorous meaning to the right-hand side of
(\ref{DefKawaOp}). Let $\mathcal{O}_c(\bbbr^d)$ denote the set
of all open sets in $\bbbr^d$ with compact closure. A
$\mathcal{B}(\Gamma)$-measurable function $F$ is called cylinder and
exponentially bounded whenever there is a
$\Lambda\in\mathcal{O}_c(\bbbr^d)$ such that
$F(\gamma)=F(\gamma_\Lambda)$ for all $\gamma \in \Gamma$, and
$|F(\gamma)|\leq C e^{c\vert\gamma_\Lambda\vert}$, $\gamma
\in\Gamma$, for some $C,c >0$. For such a function $F$ one has
\begin{eqnarray}
&&\sum_{x \in \gamma} \int_{\bbbr^d} dy\, a(x-y)
\left| F(\gamma \setminus x \cup y)- F(\gamma) \right|\nonumber \\
&\leq & \left[ |\gamma_\Lambda| \int_{\bbbr^d}\!\!\!\!\!\!dy\,a(y) +
\int_\Lambda \!\!\!\!dy\,\sum_{x \in \gamma}  a(x-y) \right] 2 C e^{c
(\vert\gamma_\Lambda\vert +1)}\label{eq2.1},
\end{eqnarray}
which is finite provided the configuration $\gamma$ is an element in
$\Gamma_a\subset\Gamma$,
\[
\Gamma_a := \left\{ \gamma \in \Gamma : y \mapsto \sum_{x \in
\gamma} a(x-y) \hbox{ is } L^{1}_\mathrm{loc} (\bbbr^d,dy)
\right\}.
\]
In this case, the sum and the integral in (\ref{DefKawaOp}) are finite, and
thus the operator $L$ is well-defined on the space
$\mathcal{F}L_{eb}^0(\Gamma_a)$ of all cylinder functions exponentially
bounded on $\Gamma$ restricted to $\Gamma_a$.

Concerning the set $\Gamma_a$, we note that $\mu(\Gamma_a)=1$ for any
probability measure $\mu$ on $\Gamma$ with first correlation function
$k^{(1)}_\mu$ bounded. This follows from the fact that for each closed ball
$B(n)\subset\bbbr^d$ centered at $0$ and radius $n\in\bbbn$ we have
\begin{eqnarray*}
\int_\Gamma\mu(d\gamma)\int_{B(n)}dy\sum_{x \in \gamma} a(x-y)
&=&\int_{\bbbr^d}dx\,k^{(1)}_\mu (x)\int_{B(n)}dy\,a(x-y)\\
&\leq& C \Vert a\Vert_{L^1(\bbbr^d,dx)}\hbox{vol}(B(n))
\end{eqnarray*}
for any constant $C\geq\vert k^{(1)}_\mu\vert$. Here $\hbox{vol}(B(n))$
denotes the volume of $B(n)$ with respect to the Lebesgue measure on
$\bbbr^d$. Clearly, the latter implies that
$\mu(\Gamma\setminus\Gamma_a)=0$.

In view of these considerations, throughout this work we
shall restrict our setting to $\Gamma_a$.

Among the elements in $\mathcal{F}L_{eb}^0(\Gamma_a)$ we distinguish
the functions $e_B(f)$, called Bogoliubov exponentials,
\begin{equation}\label{eq::defbog}
e_B(f,\gamma):=\prod_{x\in\gamma}(1+f(x)),\quad \gamma\in\Gamma,
\end{equation}
for $f$ being any bounded $\mathcal{B}(\bbbr^d)$-measurable
function with compact support ($f\in B_{c}(\bbbr^d)$). An important reason to single out these functions is due to the especially
simple form for the action of $L$ on them, namely, for all $f\in
B_{c}(\bbbr^d)$ and all $\gamma\in\Gamma_a$,
\begin{eqnarray}
\left(Le_B(f)\right)(\gamma) 
&=& \sum_{x \in \gamma} \left(A f\right)(x) e_B(f, \gamma \setminus
x), \label{Eq2.10}
\end{eqnarray}
where
\begin{equation}
\label{OnePartOp} A f(x) := \int_{\bbbr^d} dy\,a(x-y)
\left(f(y) -f(x) \right) .
\end{equation}

In the sequel we call the linear operator $A$ the pseudo-one-particle
operator, though we will explain in which way it differs from the genuine one-particle operator. Due to its special role throughout this work, its
properties will be studied in more detail in the next subsection.

Given a locally integrable function $z\geq 0$, we remind that the
Poisson measure $\pi_z$ with intensity $z$ is the unique probability
measure on $\Gamma$ for which the Laplace transform is given by
\begin{equation}
\int_\Gamma\pi_z(d\gamma )\, e^{\langle\varphi,\gamma\rangle} =\exp
\left( \int_{\bbbr^d}dx\,\left( e^{\varphi
(x)}-1\right)z(x)\right),\label{A7}
\end{equation}
for all $\varphi$ in the Schwartz space
$\mathcal{D}(\bbbr^d):=C_c^\infty(\bbbr^d)$ of all
infinitely differentiable functions with compact support. We recall
that a measure $\mu$ on $\Gamma_a$ is called infinitesimally
reversible with respect to an operator $L$, whenever $L$ is symmetric in $L^2(\Gamma_a,\mu)$.

\begin{lemma}
\label{LemRev} Assume that $a$ is an even function. Then for any
real number $z>0$ the Poisson measure $\pi_z$ with constant
intensity $z$ is an infinitesimally reversible measure with respect
to $L$.
\end{lemma}

\begin{proof}
For all $F,G\in \mathcal{F}L_{eb}^0(\Gamma_a)$ we have

\begin{eqnarray*}
\int_{\Gamma_a}\pi_z(d\gamma) LF(\gamma)G(\gamma) & =&
\int_{\Gamma_a}\pi_z(d\gamma) \sum_{x \in \gamma} \int_{\bbbr^d}dy\, a(y-x)
F(\gamma \setminus x \cup y)G(\gamma)\\ && -
\int_{\Gamma_a}\pi_z(d\gamma) \sum_{x \in \gamma} \int_{\bbbr^d}dy\, a(x-y)
F(\gamma)G(\gamma).
\end{eqnarray*}
The result follows by applying twice the Mecke identity
(cf.~\cite[Theorem~3.1]{Me67}) to the first integral.
\end{proof}

\begin{remark}
\label{Natal11} It is clear that the linear space spanned by the
class of functions $e_B(f)$, $f\in B_{c}(\bbbr^d)$, is in
$\mathcal{F}L_{eb}^0(\Gamma_a)$. The space
$\mathcal{F}L_{eb}^0(\Gamma_a)$ also contains the class of coherent
states $e_{\pi_{z}}(f)$ corresponding to $f\in
B_{c}(\bbbr^d)$,
\begin{equation}
\label{JL}
e_{\pi_{z}}(f):=\exp\left(-\int_{\bbbr^d}dx\,z(x)f(x)\right)e_B(f).
\end{equation}
see Subsection~\ref{SecFock} for a use of this relation.
\end{remark}

\subsection{The one-particle operator \label{SecQuasi}}

In the sequel the pseudo-one-particle operator $A$ introduced in
(\ref{OnePartOp}) will play an essential role. Because of this, in
this subsection we shall collect its main properties used below. We
observe that in stochastic analysis the operator $A$ is known as a
relatively simple example of a bounded generator of a Markov jump
process on $\bbbr^d$ (see e.g.~\cite[Section~4.2]{EtKu86}).

In terms of the usual convolution $*$ of functions, we also note that
\begin{equation}
Af = a*f - a^{(0)} f , \quad  \a0:=\int_{\bbbr^d}a(x)\,dx.
\label{KF1}
\end{equation}
Therefore, the properties of convolution of functions (namely,
Young's inequality) lead straightforwardly to the $L^p$ results
stated in the next proposition. There, we also consider the real
Banach spaces $B(\bbbr^d)$ and $C_\infty(\bbbr^d)$,
respectively, of all bounded measurable functions and of all
continuous functions vanishing at infinity, both with the supremum
norm $\Vert f\Vert_u:=\sup_{x\in\bbbr^d}\vert f(x)\vert$. We
recall that a strongly continuous contraction semigroup
$(T_t)_{t\geq 0}$ is called sub-Markovian whenever for $0\leq f\leq
1$ it follows $0\leq T_tf\leq 1$ for every $t\geq 0$. If,
in addition, $T_tf_n\nearrow 1$ for some sequence $f_n\nearrow 1$,
then $(T_t)_{t\geq 0}$ is called Markovian. Here, for the spaces
$B(\bbbr^d)$ and $C_\infty(\bbbr^d)$ the convergence is
pointwise, and for an $L^p$-space the convergence is almost
everywhere. A strongly continuous contraction semigroup
$(T_t)_{t\geq 0}$ is called positivity preserving, if $f\geq 0$
implies $T_tf\geq 0$ for every $t\geq 0$.

\begin{proposition}
\label{PropSemiA} The linear operator $A$ is a bounded operator on
$L^p(\bbbr^d,zdx)$, for $z>0$ constant and $1\leq p\leq\infty$
(on $B(\bbbr^d)$ and on $C_\infty(\bbbr^d)$). As a
consequence, $A$ is the generator of a uniformly continuous
semigroup $(e^{tA})_{t\geq 0}$ on $L^p(\bbbr^d,zdx)$, $1\leq
p\leq\infty$ (on $B(\bbbr^d)$ and on $C_\infty(\bbbr^d)$),
\begin{equation}
e^{tA}:= \sum_{n=0}^\infty \frac{(tA)^n}{n!},\label{KF4}
\end{equation}
the sum converging in norm for every $t\geq0$. Moreover, $(e^{tA})_{t\geq 0}$
is a positivity preserving (contraction) semigroup on each
$L^p(\bbbr^d,zdx)$ space, $1\leq p\leq\infty$  (on $B(\bbbr^d)$ and
on $C_\infty(\bbbr^d)$), Markovian on $L^p(\bbbr^d,zdx)$ for
$1\leq p<\infty$.
\end{proposition}

For the proof see e.g.~\cite[Section 4.2]{EtKu86}, \cite{J01}.

Due to (\ref{KF1}), the operator $A$, as well as the semigroup
$(e^{tA})_{t\geq 0}$, both either on a $L^p(\bbbr^d,zdx)$ space,
$1\leq p <\infty$,
or on $C_\infty(\bbbr^d)$,
may be expressed in terms of the Fourier
transform,
\[
\hat f(k):=\frac{1}{(2\pi)^{d/2}}\int_{\bbbr^d}dx\,e^{-i\langle x,  k \rangle} f(x)
\]
see e.g.~\cite{J01}.

\begin{proposition}\label{KFProp1}
For every function $\varphi$ in the Schwartz space
$\mathcal{S}(\bbbr^d)$ of tempered test functions one has
\[
\widehat{A\varphi}(k)=(2\pi)^{d/2}\left(\hat a(k)-\hat a(0)\right)\hat\varphi(k),\quad
k\in\bbbr^d,
\]
and
\begin{equation}
\left(e^{tA}\varphi\right)(x) =
\frac{1}{(2\pi)^{d/2}}\int_{\bbbr^d}dk\,
e^{i\langle k , x \rangle}e^{t(2\pi)^{d/2}(\hat a (k)-\hat a(0))}\hat\varphi(k),\quad x\in
\bbbr^d.\label{Natal5}
\end{equation}
\end{proposition}

\begin{remark}
\label{Natal6}
Since $a$ is non-negative, for all $k\in\bbbr^d$ one has
\begin{equation}
\mathrm{Re}(\hat a(k) -\hat a(0))
=\frac{1}{(2\pi)^{d/2}} \int_{\bbbr^d}dx\,(\cos (\langle k, x \rangle )-1)a(x)\leq 0.
\label{Ze}
\end{equation}
The equality in (\ref{Ze}) holds only for $k=0$. This follows from the fact
that the set $\{x: \langle k , x \rangle =2n\pi, n\in\bbbz\}$ has zero Lebesgue measure
if and only if $k\not= 0$.
\end{remark}

\begin{remark}
According to Proposition \ref{KFProp1}, the semigroup
$(e^{tA})_{t\geq 0}$ is defined by a kernel $\mu_t$, $t\geq 0$,
whose Fourier transform is given by
\[
\hat\mu_t(k)= \frac{1}{(2\pi)^{d/2}}e^{t(2\pi)^{d/2}(\hat a (k)-\hat a(0))}.
\]
Because $a$ is non-negative, $\hat a$ and thus $\hat\mu_t$, $t\geq
0$, are positive definite functions. Bochner's theorem yields that each $\mu_t$ is a non-negative finite measure on
$\bbbr^d$. We note that the Markovian property of
$(e^{tA})_{t\geq 0}$ means that each $\mu_t$ is actually a
probability measure. For more details see e.g.~\cite{J01}.
\end{remark}

\begin{remark}\label{KFRem1}
It is easy to check that on the space $L^2(\bbbr^d,zdx)$ the adjoint
operator $A^*$ of $A$ is defined by
\[
(A^* f)(x) = \int_{\bbbr^d}dy\,a(y-x) (f(y)-f(x))+ \int_{\bbbr}dy \big(a(y)-a(-y)\big)\  f(x),\quad
f\in L^2(\bbbr^d,zdx).
\]
As a consequence, if $a$ is an even function, then $A$ is a bounded
self-adjoint operator on $L^2(\bbbr^d,zdx)$. In this case, it follows
from (\ref{KF4}) that the semigroup $(e^{tA})_{t\geq 0}$ is a self-adjoint
contraction on $L^2(\bbbr^d,zdx)$.
\end{remark}

\begin{remark}
\label{remnoncons} For a non-constant activity parameter $z$, the
operator $A$ is still bounded but, in general, the semigroup
$(e^{tA})_{t\geq 0}$ is not any longer a contraction. For instance,
on $\bbbr$ consider $a(x)=e^{-x^2}$ and $z(x)= 1
+ e^{-x^2 +4x}$. A simple calculation shows that for
$\varphi(x)=\pi^{-1/2} e^{-x^2}$ one has $\int_{\bbbr}
dx\,\varphi(x) (A\varphi)(x) z(x)>0$, proving that the semigroup
$(e^{tA})_{t\geq 0}$ cannot be a contraction in
$L^2(\bbbr^d,z dx)$.
\end{remark}

\section{Independent infinite particle processes}\label{Section2}

\subsection{A probabilistic approach}\label{Subsection2.1}

\subsubsection{One particle process:}As we mentioned at the beginning of Subsection \ref{SecQuasi}, the
operator $A$ is the generator of a Markov jump process on
$\bbbr^d$. Following e.g. \cite{EtKu86}, we can explicitly
construct this process as follows. Consider the Markov chain
$(Y_k)_{k\in\bbbn_0}$ on $\bbbr^d$, with transition
density function $\mu(x,y)= \frac{a(x-y)}{\a0}$. Let $(Z_t)_{t\geq
0}$ be a Poisson process with parameter $\a0$ independent of the
Markov chain. We then define the Markov process $(X_t)_{t\geq 0}$ by
$X_t:=Y_{Z_t}$, $t\geq 0$. This process has $A$ as generator and, by
construction, it has \textit{cadlag\/} paths in $\bbbr^d$. We
denote by $D([0,\infty),\bbbr^d)$ the set of all
\textit{cadlag\/} paths from $[0,\infty)$ to $\bbbr^d$  and by
$P^x$ the path-space measure corresponding to the process $X$ starting at
$x \in \bbbr^d$. By $E_x$ we denote the expectation w.r.t.~this
measure. Hence, for $\varphi\in\mathcal{D}(\bbbr^d)$ we have that \begin{equation}
E_x[\varphi(X_t)]= e^{tA}\varphi(x).
\end{equation} 

\subsubsection{Construction:}
The process on $\Gamma$ corresponding to $L$ is the following
random evolution:  each particle evolves according to the above jump process,
independently of the other particles, cf.~Lemma \ref{nov06-eq1} below.
This independent infinite particle process was rigorously constructed in
\cite{KLR05} (see also \cite{KLR03}). In the following we present the main results therein.

We need further assumptions on the jump kernel $a$ to make this construction rigorous. It is sufficient for the construction done in \cite{KLR05}  to assume that there exist an $\alpha > 2d$ and a $C>0$ such that
\begin{equation}
\label{eqbnda} 0 \leq a(y) \leq \frac{C}{(1+|y|)^\alpha},\qquad
\text{for all } y\in \bbbr^d.
\end{equation}
Then \cite{KLR05} shows that there exists a Markov process $\left(
D([0,\infty),\Theta), (\mathbf{X}_t)_{t\geq 0},
(\mathbf{P}_\gamma)_{\gamma \in \Theta}\right)$ on the set 
\[
\Theta := \left\{ \gamma \in \Gamma \, : \, \exists m \in \bbbn \mbox{ such that }  \vert\gamma_{B(n)}\vert\leq m\,\hbox{vol}(B(n)),\quad \forall\,
n\in\bbbn \right\}. 
\]
Here $D([0,\infty),\Theta)$ denotes the set of all \textit{cadlag\/}
paths from $[0,\infty)$ to $\Theta$, the process $\mathbf{X}_t:
D([0,\infty),\Theta) \rightarrow \Theta$ is the canonical one, i.e.,
$\mathbf{X}_t(\omega):=\omega(t)$, $\omega\in D([0,\infty),\Theta)$,
and each $\mathbf{P}_\gamma$ is the path-space measure of the
process starting at a $\gamma \in \Theta$. By $\mathbf{E}_\gamma$
we denote the expectation w.r.t.~$\mathbf{P}_\gamma$.

Note, that it is impossible to let the process start from an arbitrary initial
configuration $\gamma \in \Gamma$, as follows implicitly from
the discussion before Theorem~2.2 in \cite{KLR05}.
 One is obliged to restrict the set of possible initial
configurations, in our case to $\Theta$. Hence we can start the process from configurations in $\Theta$ and initial distributions $\mu$ supported on $\Theta$.
The restriction that $\mu(\Theta)=1$ is not very severe. For example it is sufficient that for all $n$ the random variable $\vert\gamma_{B(n)}\vert$ has all exponential moments finite with respect to $\mu$ and $\vert\gamma_{B(n)}\vert/\hbox{vol}(B(n))$ has a $\mu$ integral uniformly bounded in $n$.
In particular, for any probability measure $\mu$ on $\Gamma$ whose correlation functions $k_\mu^{(n)}$,
$n\in\bbbn$, fulfill the so-called Ruelle bound, i.e., there is
a $C>0$ such that $k_\mu^{(n)}\leq C^n$ for every $n\in\bbbn$.
This holds, for instance, for Gibbs measures w.r.t.~superstable,
lower and upper regular potentials, cf.~\cite{Ru70}. Trivially, it holds for Poisson measures which have as intensity a non-negative
bounded measurable function. For a constant intensity this result
was shown using ergodicity in \cite{NZ77}.
However, finiteness of exponential moments or even finiteness of all moments are not necessary, one may assume instead decay of correlation type properties. 

\subsubsection{Starting from a configuration and transition propabilities:}
Choosing, for a fixed initial configuration $\gamma \in \Theta$, an
enumeration $\{x_n\}_{n\in\bbbn}$, the infinite particle
process can be described more explicitly. For each $n$ let us
consider an independent copy of the one-particle jump process
$( (X^{(n)}_t)_{t\geq 0}, P^{x_n} )$ introduced at the beginning
of the section. In \cite{KLR05} it is shown that if one starts in a $\gamma=\{x_n\}_{n\in\bbbn} \in \Theta$ then $\bigotimes_{n=1}^\infty P^{x_n}$-a.s.~for all times the sequence
$(X_t^{(n)})_{n\in\bbbn}$ has neither accumulation points nor two entries of the sequence coincide and hence it can be identified with the configuration $\{ X^{(n)}_t \}_{n \in
\bbbn} \in \Gamma$.  Furthermore, it has been shown that this configuration lies in $\Theta$. Hence $(\{ X^{(n)}_t \}_{n \in
\bbbn})_{t\geq 0}$ is the corresponding process on the
configuration space $\Theta$ and the path-space measure is the
``symmetrization" of $\bigotimes_{n=1}^\infty P^{x_n}$.
Moreover, the
transition probability $(\mathbf{P}_t)_{t \geq 0}$ of the process
$(\mathbf{X}_t)_{t\geq 0}$ is just the product of the one-particle
transition probabilities $e^{tA}(x-y)dy$, i.e., $\prod_{n=1}^\infty
e^{tA}(x_n-y_n) dy_n$. 

As a consequence, we can give explicit formulas for a lot of expressions. The Fourier-Laplace transform of the distribution of the processes at a fixed time, is for all non-positive
$\varphi\in\mathcal{D}(\bbbr^d)$ given by
\begin{equation}\label{eqexp}
\mathbf{E}_\gamma\left[e^{ \langle \varphi,\mathbf{X}_t \rangle }
\right] =\int_\Theta \mathbf{P}_t(\gamma,d\xi)\,e^{ \langle
\varphi,\xi \rangle } = \prod_{x \in \gamma} E_x\left[
e^{\varphi(X_t)} \right] =  \prod_{x \in \gamma} e^{tA}\varphi(x), \label{KF6a}
\end{equation}
see also~\cite{KLR05} (and also Lemma~\ref{lemapri} below for an extension to a wider class of $\varphi$),
alternatively we can use the Bogoliubov exponentials introduced in \ref{eq::defbog} for $\varphi\in\mathcal{D}(\bbbr^d)$ with
$-1< \varphi \leq 0 $,
\begin{equation}
\mathbf{E}_\gamma \left[e_B(\varphi,\mathbf{X}_t) \right] = \prod_{x
\in \gamma} E_x\left[ \varphi(X_t) +1\right] =
e_B(e^{tA}\varphi,\gamma).\label{KF6b}
\end{equation}
More generally, it follows that the Fourier-Laplace transform of the joint distribution of the processes at different times can be expressed as
\begin{eqnarray*}
\mathbf{E}_\gamma\left[e^{\langle \varphi_1,
\mathbf{X}_{t_1}\rangle}\cdots e^{\langle
\varphi_n,\mathbf{X}_{t_n}\rangle}\right]=\prod_{x \in \gamma}
E_x\left[e^{\varphi_1(X_{t_1})}\ldots e^{\varphi_n(X_{t_n})}\right],
\label{Eq3.1}
\end{eqnarray*}
for all $0\leq t_1 <\ldots < t_n $ and all non-positive
$\varphi_1,\ldots,\varphi_n\in \mathcal{D}(\bbbr^d)$, $n\geq
2$. By a monotone approximation procedure using Riemannian sums,
this relation allows us to calculate the Laplace transform of the
path-space itself measure $\mathbf{P}_\gamma$, i.e.,
\begin{equation}
\label{eqpath} \mathbf{E}_\gamma\left[ e^{-\int dt \langle
\varphi(t,\cdot), \mathbf{X}_t \rangle  } \right] = \prod_{x \in
\gamma} E_x\left[ e^{-\int dt \varphi(t, X_t)   } \right],
\end{equation}
for all non-negative continuous functions $\varphi$ from
$[0,\infty)\times \bbbr^d$ to $\bbbr$ with compact
support.

The next result gives a relation between the process $\mathbf{X}$ and the operator $L$.
\begin{lemma}
\label{nov06-eq1} For all $F \in \mathcal{F}L^0_{eb}(\Theta)$ and
all $\gamma \in \Theta$ there holds
\begin{equation}
\label{eq3.2} \mathbf{E}_\gamma \left[
F(\mathbf{X}_t)-F(\mathbf{X}_0) - \int_0^t ds LF(\mathbf{X}_s)
\right] =0.
\end{equation}
\end{lemma}

\begin{proof}
For all $F$ of the form $e_B(\varphi)$ with $\varphi \in
\mathcal{D}(\bbbr^d)$ one has
\begin{eqnarray*}
\frac{d}{dt}\mathbf{E}_\gamma\left[F(\mathbf{X}_t)\right] =
\frac{d}{dt}e_B(e^{tA}\varphi,\gamma)= \sum_{x \in \gamma}
Ae^{tA}\varphi(x) e_B(e^{tA}\varphi,\gamma \setminus x).
\end{eqnarray*}
Hence using the product structure of the path-space measure and
(\ref{Eq2.10}) one finds
\begin{eqnarray*}
 \sum_{x \in \gamma}
Ae^{tA}\varphi(x) e_B(e^{tA}\varphi,\gamma \setminus x)=
\mathbf{E}_\gamma \left[\sum_{x \in \mathbf{X}_t} A\varphi(x)
e_B(\varphi,\mathbf{X}_t\!\! \setminus x)\right]  =
\mathbf{E}_\gamma \left[Le_B(\varphi)(\mathbf{X}_t) \right].
\end{eqnarray*}

In order to make the previous calculations rigorous and to extend
them to the whole space $\mathcal{F}L^0_{eb}(\Theta)$, we
have to establish bounds for the considered expressions. According
to (\ref{eq2.1}), for all $F\in \mathcal{F}L^0_{eb}(\Theta)$, one
may bound the integrand $LF(\mathbf{X}_s)$ in (\ref{eq3.2}) by
\[
\left[ |\mathbf{X}_s\cap\Lambda| \int_{\bbbr^d}dy\,a(y) +
\int_\Lambda dy\,\sum_{x \in \mathbf{X}_s}  a(x-y) \right] 2 C e^{c
(\vert\mathbf{X}_s\cap\Lambda\vert +1)}.
\]
For some constants $C',C''>0$ one may bound the previous
expression by $C'' e^{\langle \chi, \mathbf{X_s} \rangle }$, where
$\chi(y) = \frac{C'}{(1+|y|)^{\alpha/2}}$ with $\alpha$ as in
(\ref{eqbnda}). The first
summand in (\ref{eq3.2}) and the expression $\sum_{x \in
\mathbf{X}_t} |A\varphi(x)| e_B(|\varphi|,\mathbf{X}_t \setminus x)$ can
be bounded analogously. Due to Lemma~\ref{lemapri} below the
expectation of these bounds are finite for all $\gamma \in \Theta$.
\end{proof}

\begin{lemma}
\label{lemapri} Let $\chi(y) = \frac{C'}{(1+|y|)^{\alpha/2}}$, for some $C'>0$ and $\alpha>0$ as in (\ref{eqbnda}). Then
there exists a $c>0$ such that $A\chi \leq c \chi$ and $e^{tA}\chi
\leq e^{ct} \chi$ for all $t\geq 0$. Moreover, for all $\gamma \in
\Theta$ and all measurable functions $\varphi$ such that $\frac{|\varphi|}{
\chi}$ is bounded, one has that the product $\prod_{x \in \gamma}(1+\varphi(x))$
is absolutely convergent and
\[
\mathbf{E}_\gamma \left[ e^{\langle\varphi, \mathbf{X}_t\rangle}
\right] < \infty.
\]
\end{lemma}

\begin{proof}
The previous considerations yield
\[
\mathbf{E}_\gamma \left[ e^{\langle\varphi, \mathbf{X}_t\rangle}
\right]\leq
\mathbf{E}_\gamma \left[ e^{\langle\chi, \mathbf{X}_t\rangle}
\right] = \prod_{x \in \gamma} E_x\left[ e^{\chi(X_t)}\right] =
\prod_{x \in \gamma} e^{tA}e^\chi (x).
\]
Thus the proof amounts to show the convergence of the latter
infinite product. For this purpose, we note that
\[
\chi(x-y)= \frac{C^\prime}{(1+|x-y|)^{\alpha/2}}\leq
\frac{C^\prime(1+|y|)^{\alpha/2}}{(1+|x|)^{\alpha/2}}=  (1+|y|)^{\alpha/2} \chi(x).
\]
This gives for the one-particle operator
\[
A\chi (x) = \int_{\bbbr^d} dy a(y) \chi(x-y) - a^{(0)}\chi(x) \leq  \Big(\int_{\bbbr^d} dy (1+|y|)^{\alpha/2} a(y)\Big)
\chi(x) - a^{(0)}\chi(x).
\]
Due to the bound (\ref{eqbnda}) for $a$,  the integral $\int_{\bbbr^d} dy (1+|y|)^{\alpha/2} a(y)$
is finite.
By Gr\"onwall's lemma this yields $E_x[\chi(X_t)] = e^{tA}\chi(x) \leq \chi(x)e^{ct}$,
where $c=  \int_{\bbbr^d} dy (1+|y|)^{\alpha/2} a(y) -
a^{(0)}$. Therefore, also $|e^{tA}(e^{\chi}-1)|$ decays as
$(1+|y|)^{-\alpha/2}$. According to the definition of $\Theta$, one
can prove that $\langle \chi, \gamma \rangle < \infty $ for all
$\gamma \in \Theta$. Hence also the product $\prod_{x \in \gamma}
e^{tA}e^\chi (x)$ converges absolutely and is finite.
 \end{proof}

\subsubsection{Starting from a distribution:} In Sections~\ref{led} and \ref{nonled} one considers the
one-dimensional in time distributions  of processes starting with initial
distributions $\mu$ which are probability measures on $\Theta$. The
path-space measure $\mathbf{P}^\mu$ corresponding to such a process  is
given by $\int_\Theta \mathbf{P}_\gamma \mu(d\gamma)$. Its
one-dimensional distribution is a probability measure
$P^{\mathbf{X}}_{\mu,t}$ on $\Theta$ defined for all non-negative
measurable functions $F$ by
\begin{equation}
\label{eqpromu} \int_\Theta P^{\mathbf{X}}_{\mu,t}(d\gamma) F(\gamma)
 := \int_{\Theta} \mu(d\gamma)
\mathbf{E}_\gamma\left[F(\mathbf{X}_t) \right] .
\end{equation}
In particular, for $\mu=\delta_\gamma$, $\gamma \in \Theta$, the
one-dimensional distribution coincides with the transition kernel
$\mathbf{P}_t(\gamma,\cdot)$ described above.

For functions $F$ being Bogoliubov exponentials, definition
(\ref{eqpromu}) leads to the so-called Bogoliubov functionals
\cite{Bo62}. By definition, the Bogoliubov functional
corresponding to a probability measure $\mu$ on $\Gamma$ is defined
by
\[
B_{\mu}(\varphi):= \int_\Gamma e_B(\varphi,\gamma) \mu(d\gamma),
\qquad \varphi \in \mathcal{D}(\bbbr^d).
\]
Such a functional is an analogue of the Fourier-Laplace transform on
configuration spaces, cf.~\cite{Ku03}.
 Due to (\ref{KF6b}) there is an interesting
relation between the Bogoliubov functional corresponding to the
initial distribution and the Bogoliubov functional corresponding to
the one-dimensional distribution of the process at a time $t>0$,
namely,
\begin{equation}
\label{onebog} \int_\Theta  P^{\mathbf{X}}_{\mu,t}(d\gamma) e_B(\varphi,\gamma)
= \int_\Theta \mu(d\gamma)
e_B(e^{tA}\varphi,\gamma)  = B_{\mu}(e^{tA} \varphi).
\end{equation}

In particular, for $\mu=\pi_z$ for some bounded intensity function
$z\geq 0$ one finds
\begin{eqnarray}
\int_\Theta e_B(\varphi,\gamma) P^{\mathbf{X}}_{\pi_z,t}(d\gamma) =
\int_\Theta e_B(e^{tA}\varphi,\gamma) \pi_z(d\gamma) =\exp \left(
\int_{\bbbr^d} e^{tA}\varphi(x) z(x) dx \right) \label{KF31}
\end{eqnarray}
for all $\varphi\in\mathcal{D}(\bbbr^d)$. In Section \ref{led}
we shall consider this special case in more detail.

One may ask why one does not work with the more standard Fourier-Laplace transform. Reconsidering (\ref{eqexp}) one sees that at least formally
\begin{equation}
\mathbf{E}_\gamma\left[e^{ \langle \varphi,\mathbf{X}_t \rangle }
\right]  =e^{ \langle \ln(e^{tA} e^\varphi) , \gamma \rangle }
\end{equation}
and notes that the pseudo-one-particle dynamic obtained in this way $\varphi \mapsto \ln(e^{tA} e^\varphi)$  is non-linear.

%
%
%

\subsubsection{Pseudo-one-particle operator: } We call in this paper $A$ the pseudo-one-particle operator. This is motivated by (\ref{onebog}) which gives us the possibility to study time development of the infinite particle system in terms of the semi-group $e^{tA}$. Indeed, the Bogoliubov exponentials, like the usual exponentials, form a class of function large enough to describe the dynamics uniquely. The difference with the dynamics of a system of one particle is, beside the exponential in  (\ref{onebog}), the interpretation of $z$. In the one-particle system, $z$ would describe the initial distribution of the particle and hence given by an $L^1(\bbbr^d,dx)$ function. In the infinite particle system $z$ describes the intensity of the Poisson measure, that is the density of particles and hence it is natural to consider $ z \in L^\infty(\bbbr^d,dx)$. The difference is in which space we consider the semi-group to act, and we will see in Section~\ref{led} and in Section~\ref{nonled} that the qualitative behaviour is quite different. Note that for a Poisson measure on the configuration space the number of particles is a.s. finite if and only if $z \in L^1(\bbbr^d,dx)$. 

\subsection{An analytic approach}\label{Subsection2.2}

Within the framework of infinite dimensional analysis on
configuration spaces \cite{KoKu99} one may derive alternative
representations and constructions of the dynamics. Instead of
describing the infinite particle dynamics on $\Gamma$ through the
Kolmogorov equation $\frac{\partial}{\partial t}F_t=LF_t$, the
so-called $K$-transform \cite{Le75a,Le75b} allows an alternative
description for the action of $L$ on
$\mathcal{F}L_{eb}^0(\Gamma_a)$.

Given the space
$\Gamma_0:=\{\gamma\in\Gamma:\vert\gamma\vert<\infty\}$ of finite
configurations endowed with the metrizable topology as described in
\cite{KoKu99}, let $B_{exp,ls}(\Gamma_0)$ be the space of all
exponentially bounded Borel measurable functions $G$ (i.e.,
$|G(\eta)|\leq C_1 e^{C_2 |\eta|}$ for some $C_1,C_2 > 0$) with
local support (i.e., there is a $\Lambda \in
\mathcal{O}_c(\bbbr^d)$ such that 
$G(\eta) =0$ for all $\eta\in\Gamma_0$ with $\vert\eta\cap(\bbbr^d\setminus\Lambda)\vert\not= 0$). 
The $K$-transform of a $G\in B_{exp,ls}(\Gamma_0)$
is the mapping $KG:\Gamma\to\bbbr$ defined for all
$\gamma\in\Gamma$ by $(KG)(\gamma ):=\sum_{{\eta \subset
\gamma}\atop{\vert\eta\vert < \infty} } G(\eta )$.  Note that
$K\left(B_{exp,ls}(\Gamma_0)\right)\subset\mathcal{F}L_{eb}^0(\Gamma_a)$.
Given a $G\in B_{exp,ls}(\Gamma_0)$, then in terms of the operator
$L$ we obtain
\begin{eqnarray*}
\left(L(KG)\right)(\gamma)
&=& \sum_{x \in \gamma} \int_{\bbbr^d \setminus (\gamma\setminus
x)}\!\!dy\,a(x-y)
\left[(KG)(\gamma\setminus x) + (KG (\cdot\cup y))(\gamma\setminus x)\right.\\
&&\left. -(KG)(\gamma\setminus x)-(KG (\cdot\cup x))(\gamma\setminus x)\right]
\end{eqnarray*}
which leads to the so-called symbol acting on quasi-observables
$\hat L$,
\[
(\hat{L}G)(\eta) := \sum_{x\in\eta} \int_{\bbbr^d} dy\,a(x-y)
\left(  G(\eta\setminus x \cup y )- G(\eta) \right),\quad G\in
B_{exp,ls}(\Gamma_0),
\]
and the corresponding time evolution equation $\frac
{\partial}{\partial t}G_t=\hat LG_t$. The transition kernel
$\hat{P}_t$ corresponding to $\hat L$ is then for $\eta = \{ x_1 , \ldots , x_n\}$ given by
\[
\int_{\Gamma_0} \hat{P}_t(\eta',d\eta)\,G(\eta)=
\int_{\bbbr^{dn}}  dy_1 \ldots dy_n\,
 \prod_{i=1}^n e^{tA}(x_i-y_i)  G(\{y_1,\ldots,y_n\}).
\]
This allows us to extend the explicit formula for transition kernels
of $\mathbf{X}$ to the class of all so-called observables of
additive type
\[
\int_\Theta \mathbf{P}_t(\gamma,d\xi)\,(KG)(\xi)=
K\left(\int_{\Gamma_0}
\hat{P}_t(\cdot,d\eta)\,G(\eta)\right)(\gamma).
\]

 By duality one may extend the dynamical description
to correlation functions. For this purpose, on the space
$\Gamma_0=\cup_{n=0}^\infty \{\gamma\in\Gamma:\vert\gamma\vert =
n\}$ let us consider the so-called Lebesgue-Poisson measure
$\lambda_z:=\sum_{n=0}^\infty \frac{1}{n!} (zm)^{(n)}$, where $m$ denotes the Lebesgue measure and each
$(zm)^{(n)}$, $n\in\bbbn$, is the image measure on
$\{\gamma\in\Gamma:\vert\gamma\vert = n\}$ of the product measure
$z(x_1)dx_1\cdots z(x_n)dx_n$ under the mapping
${(\bbbr^d)^n}\ni(x_1,...,x_n)\mapsto\{x_1,...,x_n\}$.
If for some probability measure $\mu$ on $\Gamma$ there is a
function $k_\mu$ on $\Gamma_0$ such that the equality
$\int_\Gamma\mu(d\gamma)\,(KG)(\gamma)=
\int_{\Gamma_0}\lambda(d\eta)\,G(\eta)k_\mu(\eta)$ holds for all
$G\in B_{exp,ls}(\Gamma_0)$. Then  $k_\mu$ are the correlation function
corresponding to $\mu$. Here we abbreviate $\lambda:= \lambda_1$.
Denoting by $\hat L^*$ the dual operator of $\hat L$ in the sense
\[
\int_{\Gamma_0}\lambda_z(d\eta)\,(\hat LG)(\eta) k(\eta)=
\int_{\Gamma_0}\lambda_z(d\eta)\,G(\eta) (\hat L^*k)(\eta),
\]
one obtains the following expression, cf \cite{FKO09} for more details,
\[
(\hat L^*k)(\eta)=\sum_{x\in\eta}\int_{\bbbr^d}dy\,a(y-x)
k(\eta\setminus x\cup y)-|\eta| k(\eta)a^{(0)}.
\]
The corresponding time evolution equation for correlation functions,
$\frac {\partial}{\partial t}k_t=\hat L^*k_t$, is the analogue of the
BBGKY-hierarchy  for the case of the free Kawasaki dynamics. In our case
this equation can be explicitly solved, namely,
\begin{equation}\label{eq::dyncorr}
k_t(\{x_1,\cdots,x_n\})=\int_{\bbbr^{dn}}dy_1\cdots dy_n\,
k_\mu(\{y_1,\cdots,y_n\})\prod_{i=1}^ne^{tA}(y_i-x_i)\label{KS}
\end{equation}
for the initial condition $k_\mu$.
Let us note that if one assumes a Ruelle bound for the initial
correlation function $k_\mu$, then all the above considerations can
be made rigorous. Indeed, $e^{tA}$ is a contraction in $L^\infty(\bbbr^d , dx)$ and hence by (\ref{eq::dyncorr}) the correlation functions retain the Ruelle bound with time. In order to show that these correlation functions correspond to a probability measure $\mu_t$ on
$\Theta \subset \Gamma$ one can either use similar arguments as in
\cite{KoKtZh06} or consider it as a moment problem on configuration space and try to verify the criteria in \cite{ik20} for a version of the infinite dimensional moment problem giving a localization on the configuration space based on techniques established in \cite{ikr14} and guarantee uniqueness using Carleman bounds, see e.g. \cite{In16}.

According to the previous considerations, the time evolution of the
particle system may also be described in terms of Bogoliubov
functionals, cf.~\cite{KoKuOl02}, through the
time evolution equation
\[
\frac{\partial}{\partial t} B_{\mu,t}(\varphi)=
\int_{\bbbr^d}dx\int_{\bbbr^d}dy\,a(x-y)(\varphi(y)-\varphi(x))
\frac{\delta B_{\mu,t}(\varphi)}{\delta \varphi(x)},\quad
\varphi\in\mathcal{D}(\bbbr^d).
\]
Here $\frac{\delta B_{\mu,t}(\varphi)}{\delta \varphi(x)}$ denotes
the first variational derivative of $B_{\mu,t}$ at $\varphi$.
Actually, $B_{\mu,t}(\varphi)$ is the Bogoliubov functional
corresponding to the one-dimensional distribution of the process
starting in $\mu$. Hence, the previous equation has an explicit solution
given by (\ref{onebog}), i.e.,
$B_{\mu,t}(\varphi)=B_{\mu}(e^{tA}\varphi)$.
If one develops the Bogoliubov functional in powers of $\varphi$ then one recover the equations for the correlation functions derived above.

\section{Equilibrium dynamics}\label{seqeq}


In this section we are interested in the representation of the
generator $L$ and its semigroup in terms of the creation,
annihilation and second quantization operators as analytic
expressions. These representations are possible because there is a
well-known canonical unitary isomorphism between the (symmetric)
Fock space and the  Poisson space. The operator theoretic side can be only made rigorous for constant activity $z$. When $a$ is symmetric then $\pi_z$ for constant activity is a reversible measure, which gave rise to the name of the section.
We start by recalling the aforementioned canonical unitary 
isomorphism. Our approach is based on \cite{AKR97} and
\cite{KSSU97}, but see also \cite{KoKuOl00} and references therein.

\subsection{The Fock space representation}
\label{SecFock}

Let $z$ be a non-negative locally integrable function, which we denote by $z(\cdot)$ in order to distinguish it from the case of constant intensity. We will explain in this sub-section why in order to use the Fock space representation to establish existence on has to assume that $z$ is constant. 

We consider the complex Hilbert
space $L^{2}(\pi_{z(\cdot)}):=L^{2}(\Gamma,\mathcal{B}(\Gamma),\pi_{z(\cdot)})$ of
square integrable complex valued functions on $\Gamma$ with respect
to the Poisson measure $\pi_{z(\cdot)}$. The coherent states introduced in
Remark~\ref{Natal11} 
generate the system of so-called Charlier polynomials, namely,
\[
e_{\pi_{z(\cdot)}}(\varphi,\gamma)=\exp\left(-\int_{\bbbr^d}dx\,z(x)f(x)\right)e_B(f)=\sum_{n=0}^{\infty}\frac{1}{n!}\langle
C_{z(\cdot)}^{n}(\gamma),\varphi^{\otimes n}\rangle,
\]
where $C_{z(\cdot)}^{n}(\gamma)\in(\mathcal{D}')^{\hat{\otimes}n}$ and $(\mathcal{D}')^{\hat{\otimes}n}$ is the $n$-th symmetric
tensor product of the Schwartz distributions space
$\mathcal{D}':= \mathcal{D}'(\bbbr^d)$. This system
is orthogonal and
 any $F\in L^{2}(\pi_{z(\cdot)})$ can be expanded in
terms of Charlier polynomials
\begin{equation}
F(\gamma)=\sum_{n=0}^{\infty}\langle
C_{z(\cdot)}^{n}(\gamma),f^{(n)}\rangle,\qquad f^{(n)}\in
L^{2}(z(\cdot)dx)^{\hat{\otimes}n}\label{mar06-eq6},
\end{equation}
where $L^{2}(z(\cdot)dx) := L^{2}(\bbbr^d, z(\cdot)dx)$.
This yields a unitary isomorphism $I_{\pi_{z(\cdot)}}$ between
$L^{2}(\pi_{z(\cdot)})$ and the so-called symmetric Fock space
\[
\mathcal{F}(L^{2}(z(\cdot)dx)):=\bigoplus_{n=0}^{\infty}n!L^{2}(z(\cdot)dx)^{\hat{\otimes}n},\qquad
L^{2}(z(\cdot)dx)^{\hat{\otimes}0}:=\bbbc.
\]
%
More precisely, for each $F\in L^{2}(\pi_{z(\cdot)})$ of the form
(\ref{mar06-eq6}) one has $I_{\pi_z(\cdot)}(F)=(f^{(n)})_{n=0}^{\infty}$.
Next we recall the definition of annihilation, creation and second
quantization operators on the total subset of Fock space vectors of
the form $ f^{(n)}=f_{1}\hat{\otimes}\ldots\hat{\otimes}f_{n}$. The
action of the  creation operator
denoted by $a^{+}(h)$, $h \in L^{2}(z(\cdot)dx)$ which acts on elements $f^{(n)}\in
L^{2}(z(\cdot)dx)^{\hat{\otimes}n}$ by $
a^{+}(h)f^{(n)}=h\hat{\otimes}f^{(n)}$. The annihilation operator $a^{-}(h)$ is the adjoint operator.

Given a semigroup $(e^{tA})_{t\geq0}$ on the Hilbert space $L^{2}(z(\cdot)dx)$ 
one can construct potentially unbounded operator
$\mathrm{Exp}(e^{tA})$ on $\mathcal{F}(L^{2}(z(\cdot)dx))$
defined by $e^{tA}\otimes...\otimes e^{tA}$ on each space
$L^{2}(z(\cdot)dx)^{\hat{\otimes}n}$. Note that whenever the operator norm of $e^{tA}$ is greater than one $\mathrm{Exp}(e^{tA})$ becomes an unbounded operator. Hence in order to be in the realm of semi-group theory we have to assume that $(e^{tA})_{t\geq0}$ is a contraction semi-group. In our case, for non-constant functions $z(\cdot)$ the semi-group $e^{tA}$ is in general
not a contraction, see Remark~\ref{remnoncons} and indeed the
previous construction gives not rise to a semi-group, cf.~Remark~\ref{remnosemi}. So the constructions below are not mathematically rigorous for non constant activity as they stand but well-defined on the coherent states, Bogoliubov exponentials respectively.

The generator of $\mathrm{Exp}(e^{tA})$ is the so-called second
quantization operator $d\mathrm{Exp}A$ corresponding to $A$.  The image of the Fock coherent state $e(f):=(f^{\otimes
n}/n!)_{n=0}^\infty$, $f\in L^{2}(z(\cdot)dx)$, under
$\mathrm{Exp}(e^{tA})$ is given by
\begin{equation}
\mathrm{Exp}(e^{tA})(e(f))=e(e^{tA}f).\label{KF7}
\end{equation}
Through the unitary isomorphism $I_{\pi_z(\cdot)}$ we obtain a contraction
semigroup $(\mathrm{Exp}_{\pi_{z(\cdot)}}(e^{tA}))_{t\geq0}$ on
$L^{2}(\pi_{z(\cdot)})$. In particular, since
$I_{\pi_z(\cdot)}^{-1}e(f)=e_{\pi_{z(\cdot)}}(f)$, it follows from (\ref{JL}) that
\begin{equation}\label{eqseconBog}
\mathrm{Exp}_{\pi_{z(\cdot)}}\left(e^{tA}\right)e_{B}(f)=e_{B}(e^{tA}f)  \exp\left( - \int_{\bbbr} (e^{tA}f(x) -f(x)) z(x) dx \right).\label{KF32}\end{equation}
The operator $L$, differentiating (\ref{eqseconBog}), can be expressed in terms of creation and
annihilation operators, namely,
\begin{equation}\label{eq4.1}
L = d\mathrm{Exp}_{\pi_{z(\cdot)}}( A ) + a^-_{\pi_{z(\cdot)}}(A^*z(\cdot)),
\end{equation}
where  $d\mathrm{Exp}_{\pi_{z(\cdot)}}( A
)$ and $a^-_{\pi_{z(\cdot)}}$ are, respectively,  the image of $d\mathrm{Exp}(A)$ and
$a^{-}$ under $I_{\pi_z(\cdot)}$, and hence 
\begin{equation}
a^-_{\pi_{z(\cdot)}}(h)F(\gamma) = \int dx z(x) \left( F(\gamma \cup x) - F(\gamma) \right).
\end{equation}
For more details see
e.g.~\cite{KoKuOl00}.

Finally, we would like to present the ``annihilation and creation
operators'' in the form more common in physical literature.  For each
$x\in\bbbr^{d}$ we define an operator $a^-(x)$ acting on
$\vec{f}=(f^{(n)})_n$ by
\[
(a^{-}(x)\vec{f})^{(n)}(y_{1},\ldots,y_{n})=\sqrt{n+1}f^{(n+1)}(x,y_{1},\ldots,y_{n}).
\]
The adjoint of the operator $a^{-}(x)$ is formally given by
\[
(a^{+}(x)\vec{f})^{(n)}(y_{1},\ldots,y_{n})=\frac{1}{\sqrt{n}}\sum_{k=1}^{n}\delta(x-y_{k})f^{(n-1)}(y_{1},\dots,\hat{y}_{k},\ldots,y_{n}),
\]
where $\hat{y}_{k}$ means that the $k$-th coordinate is excluded. Actually,
the expression for $a^{+}(x)$ is well-defined as a quadratic form. It is easy to check the
relations $a^{\pm}(f)  =  \int_{\bbbr^{d}}a^{\pm}(x)f(x)dx$  in
the sense of quadratic forms.
Rewritten in a style more similar
to the common usage in physics, the right hand side of (\ref{eq4.1}) takes the form
\[
\int dx\,z(x)\int dy \left(a(x-y)-\a0  \delta(x-y)\right) \left(
a^+_{\pi}(x)a^-_{\pi}(y) - a^-_{\pi}(y) \right),
\]
where $a_\pi^\pm(x)$ are the image of $a^\pm(x)$ under $I_{\pi_z(\cdot)}$ for more details see \cite{KoKuOl00}.

\subsection{Second quantization operator approach}\label{Aug08-eq2}

\subsubsection{The symmetric case}
If $a$ is an even function, then the generator $L$ is symmetric
(see Lemma \ref{LemRev}) and one can straightforwardly give an alternative construction to
the one presented in Section \ref{Section2} for the infinite particle dynamics. As a matter of fact,
for $a$ an even function, and for each constant $z>0$, the operator
$L$ defined in (\ref{DefKawaOp}) gives rise to a Dirichlet form on
$L^2(\Gamma,\pi_z)$,
\[
\int_\Gamma \pi_z(d\gamma)(LF)(\gamma) F(\gamma)
\!=\! -\frac{1}{2}\!\int_\Gamma\!\pi_z(d\gamma)\!\sum_{x \in \gamma}\!
\int_{\bbbr^d}\!\!dy\, a(x-y)
\left|F(\gamma\!\setminus\!x \cup y)-F(\gamma)\right|^2.
\]
This allows the use of Dirichlet forms techniques to derive a Markov
process on $\Gamma$ with \textit{cadlag\/} paths and having $\pi_z$
as an invariant measure \cite{KLRII05}. Actually, one can show that
in this situation $L$ is the second quantization operator
corresponding to the non-positive self-adjoint operator $A$ on
$L^2(\bbbr^d,zdx)$ (Remark \ref{KFRem1}). Hence, $L$ is a
non-positive essentially self-adjoint operator on $L^2(\Gamma,\pi_z)$,
and it is the generator of a contraction semigroup on
$L^2(\Gamma,\pi_z)$.

\subsubsection{The asymmetric case}

According to (\ref{KF6b}) and (\ref{JL}) the action of the
transition kernel on coherent states corresponding $\pi_{z}$ is
given by
\begin{equation}
\label{JL2} \mathbf{E}_\gamma\left[e_{\pi_z}(\varphi,\mathbf{X}_t)
\right] = e_{\pi_{z}}(e^{tA}\varphi,\gamma)
\exp\left({\int_{\bbbr^d}dx\, \left(e^{tA}\varphi(x)
-\varphi(x)\right)}z\right).
\end{equation}
This allows to express the action of the transition probability on
coherent states in terms of annihilation and creation operators
\begin{equation}
\label{SemiAnn} \mathrm{Exp}_{\pi_z}(e^{tA})e^{a^-_{\pi_{z}}(e^{tA^*}z-z)}.
\end{equation}
Observe that the right-hand side of (\ref{JL2}) gives the action of
a semigroup which preserves coherent states.
Lemma~\ref{LemClos} shows
that $L$ is the generator of a strongly continuous Markov semigroup.

\begin{lemma}\label{LemClos}
Let $\mathcal{D}_\mathrm{coh}$ be the vector space spanned by all
functions $e_B(\varphi)$ with $\varphi \in L^1(\bbbr^d,zdx)\cap
L^2(\bbbr^d,zdx)$. The operator $L$ restricted to
$\mathcal{D}_\mathrm{coh}$ is closable in $L^2(\Gamma,\pi_z)$ and its closure is an
extension of the operator $(L,\mathcal{F}L^0_\mathrm{e b}(\Gamma_a))$ defined
in Section~\ref{SecDefKa}. Moreover, it is the generator of a Markov
semigroup and when $A$ is symmetric then $ e^{tL}=\mathrm{Exp}_{\pi_z}( e^{tA})$.
\end{lemma}

The proof is an adaptation for a non-symmetric generator of the
proof of Proposition~4.1 in \cite{AKR97} and see also \cite{KoKuOl08} for details. Technically,  in our case the proof
is simpler, because $A$ is a bounded operator and we consider
functions of the type $e_B(\varphi)$.

\medskip
\begin{proof}
According to the bound (\ref{eq2.1}), one can calculate the action
of the adjoint of $(L,\mathcal{F}L^0_\mathrm{e b}(\Gamma_a))$ by the
Mecke formula, i.e.,
\[
(L^*F)(\gamma) := \sum_{x \in \gamma} \int_{\bbbr^d} dy\, \left( a(y-x)
F(\gamma \setminus x \cup y)- a(x-y)F(\gamma) \right), \quad \forall F\in
\mathcal{F}L^0_\mathrm{e b}(\Gamma_a).
\]
Hence $L^*$ is densely defined and $L$ is closable on any dense subset of
$\mathcal{F}L^0_\mathrm{e b}(\Gamma_a)$. Using again the bound (\ref{eq2.1}) one sees
that the closures of $(L,\mathcal{D}_\mathrm{coh})$ and $(L,\mathcal{F}L^0_\mathrm{e b}(\Gamma_a))$
coincide. As $A$ is a bounded operator in $L^1(\bbbr^d,zdx)$ and in
$L^2(\bbbr^d,zdx)$, the space $\mathcal{D}(\bbbr^d)$ is a
core of $A$ in $L^2$ as well as in $L^1$. Define an $L^2(\Gamma, \pi_z)$ semigroup
 on $\mathcal{D}_\mathrm{coh}$ (as an extension by continuity) of the following explicit formula
\[
T_t e_B(\varphi) := e_B(e^{tA}\varphi).
\]
Hence $\mathcal{D}_\mathrm{coh}$ is a core for the generator of $T_t$.
By computation one sees that $\varphi \mapsto e_B(\varphi)$ is a differentiable
function from $L^1(\bbbr^d,zdx)\cap L^2(\bbbr^d,zdx)$ into
$L^2(\Gamma,\pi_z)$ with derivative given by
\[
\varphi \mapsto \sum_{x \in \gamma} \varphi(x) e_B(\varphi, \gamma
\setminus x).
\]
Summarizing,  this yields that $T_t$ is a strongly continuous
contraction semigroup on $L^2(\Gamma,\pi_z)$ such that for all $\varphi \in L^1(\bbbr^d,zdx)\cap L^2(\bbbr^d,zdx)$ holds
\[
\frac{d}{dt}T_t e_B(\varphi)(\gamma) = \sum_{x \in \gamma}
Ae^{tA}\varphi(x) e_B(e^{tA}\varphi,\gamma \setminus x).
\]
For $\varphi \in \mathcal{D}(\bbbr^d)$ the r.h.s.~is just $Le_B(e^{tA}\varphi)$.
 Since $A$ is bounded this can be extended to
all $\varphi \in L^1(\bbbr^d,zdx)\cap L^2(\bbbr^d,zdx)$.
These two facts combined yield $e^{tL}=T_t$ and
$\mathcal{D}_\mathrm{coh}$ is a core of $L$. Moreover, the operator
$L$ coincides with the second quantization operator
$d\mathrm{Exp}_{\pi_z}(A)$ on $\mathcal{D}_\mathrm{coh}$. This space
is also a core for $d\mathrm{Exp}_{\pi_z}(A)$, cf.~\cite{BeKo88}.
Hence the two operators coincide.
\end{proof}

\section{Local equilibrium dynamics}\label{led}

The considerations in Section~\ref{seqeq} were essentially for Poisson measures $\pi_z$
which have as activity parameter a constant function $z$.
According to Lemma~\ref{LemRev} and Section \ref{SecDefKa}, such
measures are reversible measures for the free Kawasaki process. As a
first step towards other initial distributions, in this section we
consider the so-called local equilibrium case, that are, Poisson
measures with a non constant activity parameter. In this case that the 
activity $z$ is a varying on a scale much larger than the scale of the range of the jump kernel $a$, then in a volume which is large in comparison to the range of $a$ but small in comparison of the scale variation of $z$, 
the Poisson measure effectively has a constant activity and is effectively reversible measure for the dynamics. 
\begin{remark}
\label{remnosemi} Although for bounded non-constant functions $z$
the process starting in $\pi_z$ can be constructed explicitly,
cf.~Subsection~\ref{Subsection2.1}, one cannot expect that the
expression given in (\ref{onebog}) can be extended to a semigroup either in the
$L^1(\Gamma_a,\pi_{z})$ or in the $L^2(\Gamma_a,\pi_{z})$ sense. The
existence of the semigroup in an $L^p$-sense can only be expected
w.r.t.~an invariant measure. Indeed, this can be seen from the following
equality
\[
\frac{\| e^{tL} e_B(f) \|_{L^p(\pi_{z})}}{\| e_B(f)
\|_{L^p(\pi_{z})}} = \exp\left(\frac{1}{p} \int_{\bbbr^d} \Big(
| 1+e^{tA}f(x)|^p - | 1+f(x)|^p \Big) z(x)dx\right),
\]
and the fact that
\[
\| e^{tL} \|_{\textrm{Op},L^p(\pi_{z})} \geq \sup_{f \in
\mathcal{D}(\bbbr^d), f\geq 0,} \ \exp\left(\frac{1}{p}
\int_{\bbbr^d} \Big( (1+e^{tA}f(x))^p - ( 1+f(x))^p \Big)
z(x)dx\right),
\]
where the r.h.s.~is infinite if $(e^{tA})_{t\geq 0}$ is not a
contraction semigroup (see Remark~\ref{remnoncons}).
\end{remark}

Let $z\geq 0$ be a bounded measurable function. According to
(\ref{KF31}), for the one-dimensional distribution of the free
Kawasaki process $(\mathbf{X}_t)_{t\geq 0}$ with initial
distribution $\pi_{z}$ one finds
\begin{equation}
\label{eq6.9} \int_\Theta e_B(\varphi,\gamma)
P^{\mathbf{X}}_{\pi_z,t}(d\gamma) = \exp\left(\int_{\bbbr^d}dx\
e^{tA}\varphi(x) z(x)\right),
\end{equation}
for every $\varphi\in\mathcal{D}(\bbbr^d)$ and every $t\geq 0$.

For each $t\geq 0$ fixed, let us now consider the linear functional defined on
$L^1(\bbbr^d,dx)$ by
\[
L^1(\bbbr^d,dx)\ni f\mapsto
\int_{\bbbr^d}dx\,(e^{tA}f)(x)z(x).
\]
Due to the contractivity property of the semigroup $e^{tA}$ in
$L^1(\bbbr^d,dx)$, see Proposition \ref{PropSemiA}, and to the
boundedness of $z$, this functional is bounded on
$L^1(\bbbr^d,dx)$, and thus it is defined by a kernel $z_t\in
L^\infty(\bbbr^d,dx)$, that is,
\begin{equation}
\int_{\bbbr^d}dx\, e^{tA}f(x)z(x)=
\int_{\bbbr^d}dx\,f(x)z_t(x), \label{Natal15}
\end{equation}
for all $f\in L^1(\bbbr^d,dx)$. Moreover, since $e^{tA}$ is
positivity preserving in $L^1(\bbbr^d,dx)$, it follows from
(\ref{Natal15}) that $z_t\geq 0$.

This shows by (\ref{eq6.9}) that the one-dimensional distribution
$P^{\mathbf{X}}_{\pi_z,t}$ is just the Poisson measure $\pi_{z_t}$
(see also \cite{Do56}, \cite{D53}).

Furthermore, the path-space measure $\mathbf{P}^{\pi_z}$
corresponding to the process is also Poissonian. It is a Poisson
measure on $\Gamma_{D([0,\infty),\bbbr^d)}$ with intensity
$P^{z}$, where $P^{z}$ is the path-space measure on
$D([0,\infty),\bbbr^d)$ of the one-particle jump process
corresponding to $A$  with initial distribution $z(x)dx$. Using the
fact that the path-space measure $\mathbf{P}_\gamma$ is supported on
$D([0,\infty),\Theta)$ one easily sees that this implies that
$\mathbf{P}^{\pi_z}$ is actually supported on a subset of
$\Gamma_{D([0,\infty),\bbbr^d)}$ which can be naturally
identified with $D([0,\infty),\Theta)$.

\begin{lemma}
The path-space measure $\mathbf{P}^{\pi_z}$ of the process
$\mathbf{X}$ starting with initial distribution $\pi_{z}$ is the
Poisson measure $\pi_{P^{z}}$ on
$\Gamma_{D([0,\infty),\bbbr^d)}$. In particular, for any
continuous function $\varphi$ on $[0,\infty) \times \bbbr^d$
with compact support we have
\[
\int_{\Gamma_{D([0,\infty),\bbbr^d)}}\mathbf{P}^{\pi_z}(d\mathbf{\omega}) e^{-\int_0^\infty \langle
\varphi(t,\cdot), \mathbf{\omega}(t) \rangle dt}
 = \exp \left(
\int_{D([0,\infty),\bbbr^d)} \left[ e^{-\int_0^\infty
\varphi(t,\omega(t)) dt}-1 \right] P^{z}(d\omega) \right).
\]
\end{lemma}

\begin{proof}
By the definition of the path-space measure corresponding to the initial
measure $\pi_z$ one has
\[
\int_{\Gamma_{D([0,\infty),\bbbr^d)}} \mathbf{P}^{\pi_z}(d\mathbf{\omega}) e^{-\int_0^\infty \langle
\varphi(t,\cdot), \mathbf{\omega}(t) \rangle dt}
 = \int_\Theta \mathbf{E}_\gamma \left[
e^{-\int_0^\infty \langle \varphi(t,\cdot), \mathbf{X}_t \rangle dt}
\right] \pi_z(d\gamma).
\]
Due to (\ref{eqpath}) the latter expectation is equal to
\[
  \int_\Theta e_B(E_{\,\cdot}\left[ e^{-\int_0^\infty
\varphi(t,X_t) dt}\right]-1,\gamma) \pi_z(d\gamma) = \exp
\left(\int_{\bbbr^d} E_x\left[ e^{-\int_0^\infty \varphi(t,X_t)
dt}-1\right] z(x)dx \right),
\]
which yields the required result.
\end{proof}

\subsection{Long time asymptotic \label{subsecLargetime}}

In this subsection we want to study the behavior of the
one-dimensional distribution of the infinite particle dynamics for large
times. As mentioned above (\ref{Natal15}), the free Kawasaki dynamic
leaves the Poissonian structure of the initial distribution
unchanged and only the underlying intensity develops with time.
Thus, the analysis of the long time asymptotic behavior reduces to
the long time asymptotic of the intensity,
see Lemma~\ref{lemasympred} for details. In other words, we reduced the problem
to a pseudo-one-particle problem where the intensity plays the role of the
initial distribution.

%
%

\begin{definition}\label{defamean}
One says that a function $z\in L^1_\mathrm{loc}(\bbbr^d,dx)$
has arithmetic mean, denoted by $\mathrm{mean}(z)$, whenever
the following limit exists
\begin{equation}
\label{mean} \lim_{R\to +\infty}
\frac{1}{\mathrm{vol}(B(R))}\int_{B(R)} dx\,z(x).
\end{equation}
\end{definition}

The following proposition gives a condition on the activity under which the long time asymptotic of the one time distribution of the infinite particle dynamics exists, it is still Poissonian with constant intensity; the constant being the arithmetic mean of the initial intensity.

%

\begin{proposition}
\label{PrAsya} Let $z\geq 0$ be a bounded measurable function whose
Fourier transform is a signed measure. Then $z$ has arithmetic mean
and the one-dimensional distribution $P^{\mathbf{X}}_{\pi_z,t}$
converges weakly to $\pi_{\mathrm{mean}(z)}$ when $t$ goes to
infinity.
\end{proposition}

%
%

Under the assumptions of Proposition \ref{PrAsya}, the activity $z$
has arithmetic mean, cf.~Corollary~\ref{Colmeana}. Then, using Fourier
transform, one can easily derive the long time asymptotic for the pseudo-one-particle dynamics, cf.~Lemma~\ref{LemasyA}. By Lemma~\ref{lemasympred} this implies the existence of the long time asymptotic for the infinite particle dynamics and identifies the limit. Using the same technique one can show that the
long time asymptotic only depends on the space asymptotic of the
intensity $z$, cf.~Corollary~\ref{Corasy}.

In Proposition~\ref{PrAsya}, the assumption concerning the Fourier transform
is actually not very transparent, because we cannot reasonably restate it  in
terms of $z$ in the position variables. Therefore, we give several
examples for illustration, cf.~Example~\ref{Exmean}, and we derive certain properties of the
arithmetic mean. Some classes of activities with reasonable asymptotic behavior do not fulfill the aforementioned Fourier transform assumption, e.g.~Example~\ref{Exmean}(v).
Note that also not every bounded function has an arithmetic mean, cf.~Remark~\ref{Aug08-eq1}.

Let us discuss the properties of the arithmetic mean in some more details before we prove the lemmas mentioned above.




\begin{remark}
\label{remnoari} If two functions $z_1,z_2 \in
L^1_\mathrm{loc}(\bbbr^d,dx)$ have arithmetic mean, then for
every $\alpha_1,\alpha_2 \in \bbbr$ the function $\alpha_1 z_1
+ \alpha_2 z_2$ also has arithmetic mean and $\mathrm{mean}(\alpha_1
z_1+\alpha_2 z_2)
=\alpha_1\mathrm{mean}(z_1)+\alpha_2\mathrm{mean}(z_2)$.
\end{remark}

\begin{example}
\label{Exmean} To be more concrete we give some examples:
\begin{enumerate}
\item[(i)] If $z$ is a constant function then  $\mathrm{mean}(z) =z$.
\item[(ii)] If $z$ decays to zero, i.e., for every $\varepsilon >0 $ there exists an $R>0$ such that
$|z(x)| \leq \varepsilon$ for $x \notin B(R)$, then
$\mathrm{mean}(z)=0$.
\item[(iii)] If $z \in L^p(\bbbr^d,dx)$, $p \in [1,\infty)$, then $\mathrm{mean}(z)=0$.
\item[(iv)] Also trigonometric functions have no influence, e.g., for $d=1$ and $z(x)=1+\varepsilon\sin(x)$ we have  $\mathrm{mean}(z)=1$.
\item[(v)] A less trivial example is the following one. Given $z_0, z_1\geq
0$, let $z$ be the function defined at each
$x=(x_1,\ldots,x_d)\in\bbbr^d$ by
\[
z(x) = \left\{ \begin{array}{ll}
          z_1 & \text{if}\ x_1\geq 0 \\
          z_0 & \text{otherwise}
        \end{array} \right. .
\]
In this case one finds $\mathrm{mean}(z)=\frac{z_0+z_1}{2}$. Note
that in this case $\hat{z}$ is not a signed measure.
\end{enumerate}
\end{example}

\begin{remark}
\label{Aug08-eq1}
The arithmetic mean does not exist for all bounded non-negative
functions.
\begin{enumerate}
\item For $d=1$ and
\[
z(x) = \left\{ \begin{array}{ll}
          1 & \text{if}\ 2^{2k} \leq |x| \leq 2^{2k+1} \ \text{for a} \ k \in \bbbn_0 \\
          0 & \text{otherwise}
        \end{array} \right.,
\]
the integral in (\ref{mean}) oscillates between $1/3$ and $2/3$. Thus
$z$ does not have arithmetic mean.

\item Another example is given by the following slowly oscillating function
\[
z(x)=\cos(\ln(1+|x|))+z_0,\quad x\in\bbbr^d,
\]
where $z_0$ is a constant greater or equal to $1$. Then for large $R$ it holds
\[
\frac{1}{\mathrm{vol}(B(R))}\int_{B(R)} z(x)\,dx \sim
\frac{d}{\sqrt{1+d^2}}\ \sin\!\Big(\ln(R+1)+\arctan(d)\Big)+z_0.
\]
In general slowly varying functions will show in general spurious behavior.
\end{enumerate}
\end{remark}

\medskip

First, we prove that one may reduce the proof of Proposition \ref{PrAsya} to a pseudo-one-particle
system. This follows from the independent movement of the particles
discussed in Section~\ref{Section2}.

\begin{lemma}
\label{lemasympred} Let $0 \leq z\in
L^1_\mathrm{loc}(\bbbr^d,dx)$ be such that
\[
\lim_{t\to +\infty}\int_{\bbbr^d} dx\,e^{tA}\varphi(x)\ z(x)=:
\int_{\bbbr^d} z_\infty(dx)\varphi(x)\
\]
exists for all $\varphi \in \mathcal{D}(\bbbr^d)$. If $z_\infty$
is a (non-negative) Radon measure on $\bbbr^d$, then the one-dimensional
distribution $P^{\mathbf{X}}_{\pi_{z},t}$ converges weakly to
$\pi_{z_\infty}$ when $t$ tends to infinity.
\end{lemma}

\begin{proof} According
to (\ref{eq6.9}), for all $\varphi\in\mathcal{D}(\bbbr^d)$ we
have
\[
\int_\Theta
P^{\mathbf{X}}_{\pi_{z},t}(d\gamma)\,e^{\langle\varphi,\gamma\rangle}
=\exp\left(\int_{\bbbr^d}dx\, e^{tA}(e^\varphi-1)(x) \
z(x)\right),
\]
with $e^\varphi-1\in\mathcal{D}(\bbbr^d)$ too. By
assumption, the latter converges when $t \rightarrow \infty$ to
\[
\exp\left( \int_{\bbbr^d}z_\infty(dx)\, (e^{\varphi(x)}-1)
\right).
\]
By the definition (\ref{A7}) of a Poisson measure, observe that this
is the Laplace transform of $\pi_{z_\infty}$. The convergence of Laplace transforms implies weak convergence, see
 \cite[Theorem~4.2]{Ka76}).
\end{proof}
%

It remains to derive the long time asymptotic for the pseudo-one-particle
system.

\begin{lemma}
\label{LemasyA} Let $z\geq 0$ be a  bounded measurable function such
that its Fourier transform $\hat{z}$ is a signed measure. For all
$\varphi\in\mathcal{S}(\bbbr^d)$ one has
\begin{equation}
\label{asyA} \lim_{t\to +\infty}\int_{\bbbr^d}
dx\,e^{tA}\varphi(x)\  z(x) = (2\pi)^{-d/2}\
\hat{z}(\{0\})\int_{\bbbr^d} dx\,\varphi(x).
\end{equation}
\end{lemma}

\begin{proof}
Through the Parseval formula and the explicit formula (\ref{Natal5})
for the semigroup $(e^{tA})_{t\geq 0}$ one finds
\[
\int_{\bbbr^d} dx\, e^{tA}\varphi(x) z(x) =\int_{\bbbr^d}
\hat z(dk)\,e^{t(2\pi)^{d/2}(\hat a (-k)-\hat a(0))}
\hat\varphi(-k).
\]
Since for all $k\in\bbbr^d$, $\mathrm{Re}(\hat a(k)-\hat a(0))
\leq 0$ (cf.~Remark \ref{Natal6}), for every $t\geq 0$ we have
\[
\left|e^{t(2\pi)^{d/2}(\hat a (-k)-\hat
a(0))}\hat\varphi(-k)\right|\leq \vert\hat\varphi(-k)\vert,\quad
\forall\,k\in\bbbr^d,
\]
where $\hat\varphi\in L^1(\bbbr^d,\hat z)$. Thus, an
application of the Lebesgue dominated convergence theorem yields
\[
 \lim_{t\to \infty}\int_{\bbbr^d} dx\,e^{tA}\varphi(x) z(x)
=\int_{\bbbr^d} \hat z(dk)\,\hat\varphi(-k)\lim_{t\to
\infty}e^{t(2\pi)^{d/2}(\hat a (-k)-\hat a(0))},
\]
where
\[
\lim_{t\to \infty}e^{t(2\pi)^{d/2}(\hat a (-k)-\hat a(0))} =\left\{
\begin{array}{ll}
1, & k=0 \\
0, & k\not= 0
\end{array}
\right. .
\]
\end{proof}

The next lemma shows that the existence of the arithmetic mean is
stable under $L^1$-convergence. In particular, it yields that the
condition  assumed in Proposition~\ref{PrAsya} on the Fourier transform of the intensity is sufficient to ensure the existence of
the arithmetic mean.

\begin{lemma}
\label{Lemmeana} Let $z\geq 0$ be a bounded measurable function. Assume
that there exists a total subset of $ L^1(\bbbr^d,dx)$ and a
$C>0$ such that for all $\varphi$ in that total subset we have
\begin{equation}
\label{eqmean} \lim_{R \rightarrow \infty}R^{-d}\int_{\bbbr^d}
dx\,\varphi(x/R) z(x) = C \int_{\bbbr^d} \varphi(x) dx.
\end{equation}
Then $z$ has an arithmetic mean. In addition, equality (\ref{eqmean}) holds for
$C=\mathrm{mean}(z)$ and for any $\varphi \in L^1(\bbbr^d,dx)$.
\end{lemma}

\begin{proof}
Let $\varphi$ be an arbitrary function in $L^1(\bbbr^d,dx)$.
Then for any $\varepsilon >0$ there exists a $\psi$ in the aforementioned total set
such that $\|\varphi -\psi\|_{L^1} \leq
\varepsilon$ and $\psi$ fulfills (\ref{eqmean}). By the following
estimate, (\ref{eqmean}) also holds for $\varphi$:
\begin{eqnarray*}
\lefteqn{ \left| R^{-d}\int_{\bbbr^d} dx\,\varphi(x/R) z(x) -
C \int_{\bbbr^d}dx\varphi(x) \right| }\\
&\leq& \left| \int_{\bbbr^d} dx\, \left(\varphi(x) -
\psi(x)\right) z(xR) \right| + | C |\
\left|\int_{\bbbr^d}dx (\psi(x) - \varphi(x)) \right| \\
&&+ \left| R^{-d}\int_{\bbbr^d}
dx\,\psi(x/R) z(x) - C \int_{\bbbr^d}dx \psi(x) \right| \\
&\leq& \Big( \|z\|_u + | C | \Big) \  \|\varphi -\psi\|_{L^1} +
\left| R^{-d}\int_{\bbbr^d} dx\,\psi(x/R) z(x) - C
\int_{\bbbr^d}dx \psi(x) \right|.
\end{eqnarray*}
Hence equality (\ref{eqmean}) holds for
 any $\varphi \in L^1(\bbbr^d,dx)$.
Applying the result obtained till now to
\begin{equation}
\label{eqind} \varphi_r(x) := \left\{ \begin{array}{ll}
\frac{1}{\mathrm{vol}(B(r))}, & x \in B(r) \\
0, & \text{otherwise}
\end{array}\right. \qquad r>0
\end{equation}
yields that the l.h.s.~of (\ref{eqmean}) coincides with the
$\mathrm{mean}(z)$.
\end{proof}

\begin{corollary}
\label{Colmeana} Given a bounded measurable function $z\geq 0$ the
following two results hold:
\begin{enumerate}
\item If $z$
has arithmetic mean, then the limit in (\ref{eqmean}) exists for all $\varphi
\in L^1(\bbbr^d,dx)$ and $C=\mathrm{mean}(z)$.
\item If the Fourier transform of $z$ is a
signed measure, then $z$ has arithmetic mean and
\[
\mathrm{mean}(z)= (2\pi)^{-d/2}\ \hat{z}(\{0\}).
\]
\end{enumerate}
\end{corollary}

\begin{proof}
The first part is a direct consequence of Lemma~\ref{Lemmeana} using
the total set of all $\varphi_r$, defined as in (\ref{eqind}), and
their translates.

 If $\hat{z}$ is
a signed measure, then due to Parseval's formula for all
$\varphi\in\mathcal{S}(\bbbr^d)$ it holds
\[
R^{-d}\int_{\bbbr^d} dx\,\varphi(x/R) z(x) =
\int_{\bbbr^d} \hat z(dk)\hat\varphi(-kR) \rightarrow
\hat{z}(\{0\}) \hat{\varphi}(0),\quad R\rightarrow \infty.
\]
The limit exists because $\hat{\varphi}$ is continuous and decays
quicker than any inverse polynomial. Hence, when $R$ goes to
$\infty$, $\hat{\varphi}(-kR)$ converges to the characteristic
function of the set $\{0\}$ multiplied by $\hat{\varphi}(0)$. The family $(\hat{\varphi}(-kR))_{R\geq
0}$ is dominated by an integrable function. The second part then
follows by Lebesgue's dominated convergence theorem and an
application of Lemma~\ref{Lemmeana} for the total set
$\mathcal{S}(\bbbr^d)$.
\end{proof}

\medskip

\begin{remark}\label{rem::fourieractiv}
Let us underline the difference between $\hat{z}(\{0\})$ and the
evaluation of $\hat{z}$ as a function at zero. We demonstrate this difference for
the case when $z$ is a bounded $L^1$-function. The Fourier transform
of $z$ is then a continuous function which we denote for the moment
by $\tilde{z}$. Evaluation at zero makes sense in this case.
However, we want to consider the Fourier transform as a generalized
function, i.e., as a linear form on a function space. We assume that
this linear form is regular enough to be expressed by a signed
measure. If, as in our case,  $z$ is an $L^1$-function, then its Fourier transform
is the measure $\hat{z}(dk)=\tilde{z}(k)dk$. Thus  $\hat{z}(\{0\})=0$, which
is totally different from $\tilde{z}(0)=\int_{\bbbr^d}dx z(x)$.

Below we list the Fourier transforms (interpreted as generalized
functions, signed measures, respectively) of the examples given in
Examples~\ref{Exmean}, provided they exist and with the same
enumeration:
\begin{enumerate}
\item[(i)] $(2\pi)^{d/2}z\delta_0(dk)$.
\item[(iii)] a continuous function for $p=1$ or an $L^{p/(p-1)}$-function for $1<p \leq 2$ multiplied in both cases by the Lebesgue measure.
\item[(iv)] $(2\pi)^{d/2} \left( \delta_0(dk) + i\varepsilon/2 \delta_1(dk) -
i\varepsilon/2 \delta_{-1}(dk)\right)$.
\item[(v)] In the one dimensional case the Fourier transform is the following generalized function
$\hat{z}(k)= \sqrt{2\pi}(z_0+z_1)/2\delta_0(k)
+i(z_0-z_1)/\sqrt{2\pi} \mathcal{P}(1/k)$, where $\mathcal{P}(1/k)$
denotes the Cauchy principal value of $1/k$. Using this explicit
formula the conclusion of Proposition~\ref{PrAsya} can be shown
although the assumptions of Proposition~\ref{PrAsya} are not
fulfilled.
\end{enumerate}
\end{remark}

Applying the same technique as in Lemma~\ref{LemasyA} we can prove
that the time asymptotic depends only on the behavior of $z$ at
infinity.

\begin{corollary}
\label{Corasy} Let $z_1,z_2\geq 0$ be two bounded measurable
functions. If $z_1 - z_2\in L^p(\bbbr^d,dx)$ for some $1 \leq p \leq 2$, then the free
Kawasaki dynamics with initial distribution $\pi_{z_1}$ and the free
Kawasaki dynamics
with initial distribution $\pi_{z_2}$ have the same long time
asymptotic limit.
\end{corollary}

\medskip

\begin{proof}
For all $\varphi\in\mathcal{D}(\bbbr^d)$ it follows from
Lemma~\ref{LemasyA} and Corollary~\ref{Colmeana} that
\begin{eqnarray*}
\int_{\bbbr^d}dx\, e^{tA}(e^\varphi-1)(x) z_1(x) -
\int_{\bbbr^d}dx\, e^{tA}(e^\varphi-1)(x) z_2(x)
\end{eqnarray*}
converges when $t$ goes to $\infty$ to
\begin{eqnarray*}
(2\pi)^{-d/2}\widehat{(z_1-z_2)}(\{0\}) \int_{\bbbr^d} (e^{\varphi(x)}-1)
dx.
\end{eqnarray*}
As we discussed in Example \ref{Exmean}, the assumption on $z_1-
z_2$ implies that $\mathrm{mean}(z_1-z_2) =0$ and $\widehat{z_1 - z_2}$ is a signed measure. Hence
\[
\lim_{t\to +\infty}\int_\Gamma
P^{\mathbf{X}}_{\pi_{z_1,t}}(d\gamma)\,e^{\langle\varphi,\gamma\rangle}=
\lim_{t\to +\infty}\int_\Gamma
P^{\mathbf{X}}_{\pi_{z_2,t}}(d\gamma)\,e^{\langle\varphi,\gamma\rangle}.
\]
By \cite[Theorem~4.2]{Ka76}), this is enough to show the required result.
\end{proof}

\medskip

Throughout this subsection, the study of the time asymptotic
behavior of the free Kawasaki process with an initial distribution
$\pi_z$ was based on the analysis of the Laplace transform
\begin{equation}
 \int_\Gamma
\pi_z(d\gamma)\,\mathbf{E}_\gamma[e^{\langle\varphi,\mathbf{X}_t\rangle}],\label{Natal1}
\end{equation}
cf.~Lemma~\ref{lemasympred}. Another way to interpret this is as follows: we have studied the time
asymptotic behavior of the so-called empirical field corresponding
to a $\varphi\in\mathcal{D}(\bbbr^d)$,
$$n_t(\varphi,\mathbf{X}) := \langle\varphi,\mathbf{X}_t\rangle=
\sum_{x \in \mathbf{X}_t}\varphi(x)$$ 
and showed that the limit $t \rightarrow \infty$ exists and identified the limiting distribution. The special role of the
empirical fields is essential in the next subsection.

\subsection{Hydrodynamic limits\label{SubSecHydro}}

In the sequel let $z\geq 0$ be a bounded measurable function. In
order to obtain a macroscopic description of our system, we rescale
simultaneously the empirical field $n_t(\varphi,\mathbf{X})=\langle
\varphi, \mathbf{X}_t \rangle$,
$\varphi\in\mathcal{D}(\bbbr^d)$, in space and in time. The
scale transformation in space is given by
$\langle\varphi,\gamma\rangle\to\varepsilon^d\langle\varphi(\varepsilon\cdot),\gamma\rangle$,
and in time by $t\to\varepsilon^{-\kappa}t$ for some $\kappa
>0$. To obtain non-trivial macroscopic density profiles, one has to
scale the initial intensity as well, $z\to z(\varepsilon\cdot)$.
This scaling yields a scaling of the Laplace transform of the
empirical field, in other words, the Laplace transform of the
one-dimensional distribution of the scaled process. For each $t\geq
0$ and each $\varphi\in\mathcal{D}(\bbbr^d)$ we obtain from
(\ref{KF6b}) the following form
\begin{eqnarray}
&&\int_\Gamma \pi_{z(\varepsilon\cdot)}(d\gamma)\,
\mathbf{E}_\gamma\left[e^{\varepsilon^d\left\langle\varphi(\varepsilon\cdot),\mathbf{X}_{\varepsilon^{-\kappa}t}\right\rangle}\right]\nonumber \\
&=&\int_\Gamma \pi_{z(\varepsilon\cdot)}(d\gamma)\,
e_B\left(e^{\varepsilon^{-\kappa}tA}\left(e^{\varepsilon^d\varphi(\varepsilon\cdot)} -1\right),\gamma\right)\label{Natal2} \\
&=&\exp\left( \int_{\bbbr^d} dx\,
\left(e^{\varepsilon^{-\kappa}tA}\left(e^{\varepsilon^d\varphi(\varepsilon\cdot)}-1\right)\right)(x)
z(\varepsilon x)\right).\nonumber
\end{eqnarray}
 In the sequel we
denote the scaled empirical field by
\begin{equation}
\label{eqReHy} n_t^{(\varepsilon)}(\varphi,\mathbf{X}) :=
n_{\varepsilon^{-\kappa}t}(\varepsilon^d\varphi(\varepsilon
\cdot),\mathbf{X})=
\varepsilon^d\left\langle\varphi(\varepsilon\cdot),
\mathbf{X}_{\varepsilon^{-\kappa}t}\right\rangle.
\end{equation}
According to the independent movement of the particles, we are again
able to reduce the study of the infinite particle system to an
effective pseudo-one-particle system. Again technical difficulties arise
from the fact that the scale of the system size is much larger than
the scale of space and time considered in the empirical field.
Actually, the system size is infinite. Under the additional
assumption that the  Fourier transform of the activity $z$ is a
signed measure, the hydrodynamic limit can be derived rather directly
using Fourier techniques, cf.~Propositions~\ref{PropHyd} and
\ref{PropHydW}. The general case of just bounded activities requires
more technical involved considerations (postponed to
Proposition~\ref{Prohydrogen}).

\begin{proposition}
\label{PropHyd}
Let $z\geq 0$ be a bounded measurable function such that its Fourier transform
is a signed measure. For each $t\geq 0$ the following limit exists for all
$\varphi\in\mathcal{D}(\bbbr^d)$
\begin{equation}
\label{EqHydLap} \lim_{\varepsilon\to
0^+}\int_\Theta\pi_{z(\varepsilon \cdot)}(d\gamma)\,
\mathbf{E}_\gamma\left[e^{n^{(\varepsilon)}_t(\varphi,\mathbf{X})}\right]
= \int_{\mathcal{D}'(\bbbr^d)} \delta_{\rho_t}(d\omega)
e^{\langle \varphi, \omega \rangle}
\end{equation}
whenever one of the following conditions is fulfilled:
\begin{enumerate}
\item If
\[
a^{(1)}_i := \int_{\bbbr^d}dx\,x_i a(x)<\infty,\quad \forall\,i=1,\cdots ,d,
\]
and $a^{(1)}:=(a^{(1)}_1,\cdots,a^{(1)}_d)\neq 0$, then for
$\kappa=1$ the limiting density $\rho_t$ is given by
\[
\int_{\bbbr^d}dx\,\rho_t(x)\varphi (x) :=
\int_{\bbbr^d}dx\,z(x + ta^{(1)})\varphi (x), \quad \varphi \in
\mathcal{D}(\bbbr^d);
\]
\item If $a^{(1)}= 0$, and
\[
a^{(2)}_{ij} := \int_{\bbbr^d}dx\,x_i x_j a(x)<\infty,\quad \forall\,i,j=1,\cdots ,d,
\]
then for $\kappa=2$ the limiting density $\rho_t$ is given by
\[
\int_{\bbbr^d}dx\,\rho_t(x)\varphi (x)
:=\frac{1}{(2\pi)^{d/2}}\int_{\bbbr^d}dx\,z(x)
\int_{\bbbr^d}dk\,e^{i\langle k ,  x \rangle } e^{-\frac{t}{2}\langle
a^{(2)}k,k\rangle}\hat\varphi (k),  \quad \varphi \in
\mathcal{D}(\bbbr^d)
\]
where $a^{(2)}$ denotes the $d\times d$ matrix with coefficients
$a^{(2)}_{ij}$.
\end{enumerate}
\end{proposition}

\begin{proof}
By (\ref{Natal2}), for each $t\geq 0$ and each
$\varphi\in\mathcal{D}(\bbbr^d)$, one has
\begin{eqnarray}
&&\int_\Gamma\pi_{z(\varepsilon
\cdot)}(d\gamma)\,\mathbf{E}_\gamma\left[
e^{n^{(\varepsilon)}_t(\varphi,\mathbf{X})}\right]\nonumber \\
&=& \exp\left( \int_{\bbbr^d} dx\,
e^{\varepsilon^{-\kappa}tA}(e^{\varepsilon^d\varphi(\varepsilon\cdot)}-1)(x)\  z(\varepsilon x)\right)\nonumber \\
&=&\exp\left( \int_{\bbbr^d} dx\,
e^{\varepsilon^{-\kappa}tA}\Big(e^{\varepsilon^d\varphi(\varepsilon\cdot)}-1
-\varepsilon^d\varphi(\varepsilon\cdot)\Big)(x)\ z(\varepsilon x)\right)\cdot\label{Natal4}\\
&&\cdot\exp\left(\int_{\bbbr^d} dx\,\varepsilon^dz(\varepsilon
x)
\left(e^{\varepsilon^{-\kappa}tA}\varphi(\varepsilon\cdot)\right)\left(x
\right)\right).\label{Natal3}
\end{eqnarray}
Concerning (\ref{Natal4}), we observe that due to the contractivity
property of the semigroup $(e^{tA})_{t\geq 0}$ in
$L^1(\bbbr^d,dx)$ (cf.~Proposition \ref{PropSemiA}) we have
\begin{eqnarray}
&& \left|\int_{\bbbr^d}dx\, e^{\varepsilon^{-\kappa}tA}
\Big(e^{\varepsilon^{d}\varphi(\varepsilon \cdot)} -1 -
\varepsilon^d\varphi(\varepsilon \cdot)\Big)( x)
z(\varepsilon x )\right| \label{Eq6.2}\\
&\leq&\| e^{\varepsilon^d\varphi(\varepsilon \cdot)} -1 -
\varepsilon^d\varphi(\varepsilon \cdot)\|_{L^{1}(\bbbr^d,dx)}
\|z \|_u  \nonumber \\
&\leq& \varepsilon^d\| \varphi\|^2_{L^1(\bbbr^d,dx)} e^{\|
\varphi\|_{L^{1}(\bbbr^d,dx)}}\|z \|_u, \nonumber
\end{eqnarray}
for more details see Lemma \ref{lemmHyd} in the
Appendix~\ref{appest} below. Thus, the proof reduces to check the
existence of the limit of the exponential (\ref{Natal3}) when
$\varepsilon$ converges to 0. For this purpose one shall apply the
Parseval formula to the exponent in (\ref{Natal3}). The use of the
explicit formula (\ref{Natal5}) for the semigroup $(e^{tA})_{t\geq
0}$ then yields
\begin{equation}
\label{eqonefour} \int_{\bbbr^d}\hat
z(dk)\,e^{\varepsilon^{-\kappa}t(2\pi)^{d/2}(\hat{a}(-\varepsilon
k)-\hat{a}(0))} \hat\varphi (-k).
\end{equation}
Here, observe that since for all $k\in\bbbr^d$,
$\mathrm{Re}(\hat a(k)-\hat a(0)) \leq 0$ (cf.~Remark \ref{Natal6}),
one has
\[
\left|e^{\varepsilon^{-\kappa}t(2\pi)^{d/2}(\hat a(-\varepsilon
k)-\hat a(0))} \hat\varphi (-k)\right|\leq
\left|\hat\varphi(-k)\right|,\quad \forall\,k\in\bbbr^d,
\varepsilon >0,
\]
where, in particular, $\hat\varphi\in L^1(\bbbr^d,\hat z)$.
Therefore, due to the Lebesgue dominated convergence theorem, one
may infer the existence of the required limit from the existence of
the limit
\begin{equation}
\label{eq6.10} \lim_{\varepsilon\to 0^+}\frac{\hat a(-\varepsilon
k)-\hat a(0)}{\varepsilon^\kappa}
\end{equation}
for each $k\in\bbbr^d$.

\medskip

\noindent Case 1: Let $\kappa=1$. Then, under the assumptions stated
in 1, the function $\hat a$ is differentiable at 0, and thus
(\ref{eq6.10}) exists with
\[
\lim_{\varepsilon\to 0^+}\frac{\hat{a}(-\varepsilon
k)-\hat{a}(0))}{\varepsilon} = \frac{i}{(2\pi)^{d/2}} \langle k , a^{(1)} \rangle
\]
for each $k\in\bbbr^d$. Hence the limit (\ref{Natal3}) exists,
being equal to
\begin{eqnarray*}
\int_{\bbbr^d}\hat z(dk)\,e^{it \langle k,  a^{(1)}\rangle }\hat\varphi (-k)
&=&\int_{\bbbr^d}dx\,z(x + ta^{(1)})\varphi (x), \quad \varphi
\in \mathcal{D}(\bbbr^d).
\end{eqnarray*}
Therefore, the Laplace transform of the rescaled empirical field
converges to
\[
\exp\left(\int_{\bbbr^d}dx\,z(x + ta^{(1)})\varphi (x)\right).
\]
This means that the limiting
distribution is the Dirac measure with mass at $z(x + ta^{(1)})$.

 \noindent Case 2: Now let $\kappa=2$. Then (\ref{eq6.10}) can be
 written as
\begin{eqnarray*}
\frac{\hat{a}(-\varepsilon k)-\hat{a}(0)}{\varepsilon^2}
&=&\frac{1}{(2\pi)^{d/2}}\int_{\bbbr^d}dx\,a(x)
\left(\frac{i}{\varepsilon} \langle x ,  k \rangle  -
\frac{1}{2}\sum_{i,j=1}^dx_ix_jk_ik_j
+ \frac{o(\varepsilon^2)}{\varepsilon^2}\right)\\
\end{eqnarray*}
where we have used the Taylor expansion of the function $e^{i\langle x,
k\rangle }$ at $x=0$ and the assumptions for the Case~2. Since $a\in
L^1(\bbbr^d,dx)$, the latter converges to
$-1/2(2\pi)^{-d/2}\langle a^{(2)}k,k\rangle$, and thus as $\mathrm{Re}(\hat a(k) -\hat a(0))\leq 0$ by (\ref{Ze}) we obtain the result. 
\end{proof}

\begin{remark}
\label{Natal8} The limiting
density $\rho_t(x)=z(x + ta^{(1)})$ obtained in Proposition
\ref{PropHyd} is the solution of the linear partial
differential equation $\frac{\partial}{\partial
t}\rho_t(x)=\langle a^{(1)} ,\nabla\rho_t(x) \rangle =
\mathrm{div}(a^{(1)}\rho_t(x))$ with the initial condition
$\rho_0=z$. In the same way, the second case stated in Proposition
\ref{PropHyd} yields a limiting density which is the solution
of the heat equation
\[
\frac{\partial}{\partial
t}\rho_t(x)=\frac12\sum_{i,j=1}^{d}a_{ij}^{(2)}
\frac{\partial^2}{\partial x_i\partial x_j}\rho_t(x)
\]
with the same initial condition.
\end{remark}

Given the function $a$, we may decompose $a$ into a sum of an even
function $p$ and an odd function $q$, $a=p+q$. Note that one always
has $p^{(1)}= 0$. Submitting also the function $a$ itself to a proper scale
transformation (beyond the previous scale transformations in space
and in time), one may complement the statement of Proposition
\ref{PropHyd}. The limit (\ref{eq6.10}) required in Proposition
\ref{PropHyd} suggests the scaling $a_\varepsilon:=p+\varepsilon q$
and $\kappa=2$. Note that the assumptions on the function $a$ (i.e.,
$0\leq a\in L^1(\bbbr^d,dx)$), carry over to $a_\varepsilon$,
$\varepsilon
>0$. Hence all considerations done in Section \ref{SecDefKa} and
 following sections for the one-particle operator $A$ and its
underlying dynamics still hold for the operator $A_\varepsilon$
defined by (\ref{OnePartOp}) with $a$ replaced by $a_\varepsilon$.
Denote by $$a^{(1)} := \int_{\bbbr^d}dx\,x a(x) =
\int_{\bbbr^d}dx\,x q(x)=
\varepsilon^{-1}\int_{\bbbr^d}dx\,x a_\varepsilon(x) .$$

\begin{proposition}(``weak asymmetry'')
\label{PropHydW} Let $z\geq 0$ be a bounded measurable function such
that its Fourier transform is a signed measure. Under the above
conditions, if $0\not= q^{(1)}=a^{(1)}\in \bbbr^d$ and
$p^{(2)}_{ij}<\infty$ for every $i,j=1,\cdots ,d$, then for each
$t\geq 0$ and each $\varphi\in\mathcal{D}(\bbbr^d)$ the
following limit exists, and it is given by
\begin{eqnarray*}
&&\lim_{\varepsilon\to 0^+}\int_\Gamma\pi_{z(\varepsilon \cdot)}(d\gamma)\,
\mathbf{E}_\gamma[e^{n^{(\varepsilon)}_t(\varphi,\mathbf{X})}]\\
&=&\exp\left(\frac{1}{(2\pi)^{d/2}}\int_{\bbbr^d}dx\,z(x)
\int_{\bbbr^d}dk\,e^{i \langle k, x\rangle } e^{-it\langle a^{(1)},k\rangle
- \frac{t}{2}\langle a^{(2)}k,k\rangle}\hat\varphi (k)\right).
\end{eqnarray*}
\end{proposition}

\begin{proof} Since all statements for the operator $A$ and the corresponding process
also hold for $A_\varepsilon$, in particular, one finds that
$(e^{tA_\varepsilon})_{t\geq 0}$ is also a contraction semigroup on
$L^1(\bbbr^d,dx)$. Similar arguments as in the proof of
Proposition~\ref{PropHyd} reduce the proof to the analysis of
existence of the limit
\[
\lim_{\varepsilon\to 0^+}\frac{\widehat{a_\varepsilon}(-\varepsilon
k)- \widehat{a_\varepsilon}(0)}{\varepsilon^2}= \lim_{\varepsilon\to
0^+}\frac{\hat p(-\varepsilon k)-\hat p(0)}{\varepsilon^2}
+\lim_{\varepsilon\to 0^+}\frac{\hat q(-\varepsilon k)-\hat
q(0)}{\varepsilon},\quad k\in\bbbr^d.
\]
\end{proof}

\begin{remark}
\label{remweaka} Similarly to Remark~\ref{Natal8}, one may then
conclude that for continuous differentiable $z$ Proposition
\ref{PropHydW} leads to a limiting density $\rho_t$ which is a
solution of the partial differential equation
\[
\frac{\partial}{\partial t}\rho_t(x)=\mathrm{div}(a^{(1)}\rho_t(x))+
\frac12\sum_{i,j=1}^{d}a_{ij}^{(2)}\frac{\partial^2}{\partial
x_i\partial x_j}\rho_t(x)
\]
 with the initial condition $\rho_0=z$.
\end{remark}

\begin{proposition}
\label{Prohydrogen} Assume that $a$ has all moments finite. Then the
results stated in Propositions~\ref{PropHyd} and \ref{PropHydW} hold
for all non-negative bounded measurable functions $z$.
\end{proposition}

\begin{proof}
We note that in the first part of the proof of
Propositions~\ref{PropHyd} and \ref{PropHydW} we have only used the
boundedness of $z$. Therefore, it remains to prove the convergence
of the exponent in (\ref{Natal3}), namely,
\begin{eqnarray}
\label{eq6.7} &&\int_{\bbbr^d} dx\,z(x)
\left(e^{\varepsilon^{-\kappa}tA_\varepsilon}\varphi(\varepsilon\cdot)\right)\left(\frac{x}{\varepsilon}\right).
\end{eqnarray}
Let $\psi\in L^1(\bbbr^d,dx)$ be given. Then
\begin{eqnarray*}
&&\left| \int_{\bbbr^d} dx\,z(x)
\left(e^{\varepsilon^{-\kappa}tA_\varepsilon}\varphi(\varepsilon\cdot)\right)\left(\frac{x}{\varepsilon}\right)
-\int_{\bbbr^d} dx\,z(x)
\left(e^{\varepsilon^{-\kappa}tA_\varepsilon}\psi(\varepsilon\cdot)\right)\left(\frac{x}{\varepsilon}\right)\right|\\
& \leq &\varepsilon^d \int_{\bbbr^d} dx\,|z(\varepsilon x)|
\left| e^{\varepsilon^{-\kappa}tA_\varepsilon} \left(
\varphi(\varepsilon\cdot)- \psi(\varepsilon\cdot)\right)(x)\right|\\
& \leq &\|z\|_u \varepsilon^d \int_{\bbbr^d} dx\left|
 \varphi(\varepsilon x)-
\psi(\varepsilon x)\right| = \|z\|_u \|\varphi
-\psi\|_{L^1(\bbbr^d,dx)},
\end{eqnarray*}
using that $(e^{tA_\varepsilon})_{t\geq 0}$ is a
$L^1(\bbbr^d,dx)$-contraction semigroup.

Therefore, it is enough to consider (\ref{eq6.7}) for $\varphi$ from
a total subset of $L^1(\bbbr^d,dx)$. Let us consider the set of
all functions in $L^1(\bbbr^d,dx)$ with Fourier
transform in $\mathcal{D}(\bbbr^d)$. This set is total in
$L^1(\bbbr^d,dx)$, because $\mathcal{D}(\bbbr^d)$ is dense
in $\mathcal{S}(\bbbr^d)$, the Fourier transform is continuous
in $\mathcal{S}(\bbbr^d)$, and $\mathcal{S}(\bbbr^d)$ is
dense in $L^1(\bbbr^d,dx)$. Let $\varphi$ be such a function.
In order to prove the convergence of (\ref{eq6.7}) it is enough to
show the convergence of
$(e^{\varepsilon^{-\kappa}tA_{\varepsilon}}\varphi(\varepsilon\cdot))(\frac{x}{\varepsilon})$
in $L^1(\bbbr^d,dx)$. For this purpose it is sufficient to show the following
convergence of the Fourier transform of the latter expression,
\begin{equation}
\label{eq6.8}
e^{\varepsilon^{-\kappa}t(2\pi)^{d/2}(\hat{a}_\varepsilon(-\varepsilon
k)-\hat{a}(0))} \hat\varphi (-k) \rightarrow e^{ it \langle
a^{(1)},k \rangle  -\frac{0^{2-\kappa}}{2} t \langle
a^{(2)}k,k\rangle } \hat\varphi (-k),\quad \varepsilon \rightarrow 0
\end{equation}
in $\mathcal{S}(\bbbr^d)$, actually in the Sobolev space $H^k(\bbbr^d)$ with $k > d/2$ is sufficient.
The exponent in the l.h.s. of (\ref{eq6.8}) may be
written as
\begin{eqnarray}
i  t \int_{\bbbr^d} \langle k, x \rangle q(x) dx
-\varepsilon^{2-\kappa} t \int_0^1 (1-s)
\int_{\bbbr^d} \langle k,x\rangle^2  e^{is\varepsilon \langle
k,x\rangle } a_\varepsilon(x) dx ds. \label{Eq6.1}
\end{eqnarray}
Hence no terms with negative powers in $\varepsilon$ appears. The same will be true for all derivatives. As $\mathrm{Re}(\hat{a_\varepsilon} - \hat{a}(0))$ is non-positive all exponentials in (\ref{eq6.8}) can be bounded by one. Hence all derivatives of the exponential in (\ref{eq6.8}) will be bounded by a polynomial in $k$ which has degree the order of the derivative considered. This bound will be uniform in $\varepsilon$. As $\hat\varphi$ decays quicker than any polynomial we obtain the desired convergence.
\end{proof}

\subsection{Finite volume system}\label{subsecfv}

The situation is much easier when one starts with a particle system on a compact manifold. Let us explain this in the case, when we consider particle configurations on $[-L/2,L/2]^d$ with periodic boundary conditions, that is we consider the system on the torus. Let us start by giving a description of the system in global coordinates coming from the covering space of the torus. For any sufficiently decaying function $\varphi$ on $\bbbr^d$, denote by $\varphi_L(x) := \sum_{j \in \bbbz^d} \varphi(jL)$. Any function on the torus can be written in this way. Consider the particle system $\Gamma_L := \left\{ \gamma \subset [-L/2, L/2]^d \, : \, | \gamma | < \infty \right\}$ with the jump rate given by $a_L$. The probabilistic construction is the finite product and hence automatically a configuration at least if one allows coinciding points. The particle dynamics extend to this case and can also be reduced in the same way to a pseudo-one-particle dynamics. Denote by $A_L$ the generator of the latter. An easy calculation show that $A_L\varphi_L =(A \varphi)_L$
 and hence $e^{tA_L}\varphi_L =(e^{tA} \varphi)_L$. 
 Therefore, the semi-group of the pseudo-one-particle on the torus can be described in the following way
\begin{equation}\label{eqsemperiod}
\int_{[-L/2,L/2]^d} \left(e^{tA_L}\varphi_L \right)(x)\, z(x) dx = \int_{\bbbr^d} (e^{tA} \varphi)(x) \, P_Lz(x) dx,
\end{equation}
where 
\begin{equation}
P_Lz(x) := \sum_{j \in \bbbz} \Ii_{[-L/2,L/2]^d}(x-jL) z(x-jL)
\end{equation}
is the extension of $z$ to a $(L\bbbz)^d$ periodic function. Note that if $\varphi$ has compact support then $\varphi_L= \varphi$ for $L$ large enough.
In order to connect with our previous consideration we take the Fourier transform and obtain the expression
\begin{equation}\label{eqonepartffv}
(\sqrt{2\pi}L)^{-d} \sum_{j \in \bbbz^d}  e^{t(2\pi)^{d/2}(\hat a (-j/L)-\hat a(0))}
\hat\varphi(-j/L) \widehat{ \Ii_{[-L/2,L/2]^d}}*\hat{z}(j/L).
\end{equation}
\subsubsection{Long time asymptotic:} The behaviour at large time for general $z$ can be easily established in the finite volume. As $\widehat{ \Ii_{[-L/2,L/2]^d}}*\hat{z}$ is a continuous function one sees directly that for fixed $L$, the above converges to $$(\sqrt{2\pi}L)^{-d}  \widehat{ \Ii_{[-L/2,L/2]^d}}*\hat{z}(0) = \frac{1}{L^d} \int_{[-L/2,L/2]^d} z(x) dx $$ with an exponential rate. The latter in turn converges to $\mathrm{mean}(z)$ for $L \rightarrow \infty$. We only used that $a \in L^1$ and that $z \in L^\infty$ such that the arithmetic mean exists in contrast to the extra regularity required in Lemma~\ref{LemasyA}, namely $\hat{z}$ is a signed measure. 

We showed that the convergence is equivalent to the existence of the arithmetic mean without any regularity assumption of $z$ except being bounded. When one first sends $L$ to infinity then (\ref{eqsemperiod}) converges to (\ref{Natal15}) and then considering the time asymptotic afterwards can become much more involved. In the finite volume case one has a spectral gap proportional to $L^{-2}$, which will disappear in the $L \rightarrow \infty $. Let us first consider what happens when one sends $t$ and $L$ simultaneously to infinity but $t$ is sufficiently larger than $L$. We need first a result to control the spectral behaviour of the semi-group near $k=0$.

\begin{lemma}\label{lemboundeaTK}
Assume that $a$ has finite second moments and that the matrix $(a_{i,j}^{(2)})_{i,j=1, \ldots , d}$ is positive definite in the sense that there exists a $\sigma' >0$ such that $
\sum_{i,j=1}^d k_i k_j a_{i,j}^{(2)} \geq \sigma^{-2} |k|^2$ for all $k \in \bbbr^d$.
Then for each $\sigma > \sigma'$ there exists a constant $c >0$ such that
\begin{equation}\label{boundexpTK}
\left| e^{t(2\pi)^{d/2} (\hat a (-k)-\hat a(0))} \right|   \leq \exp\left( -t  \frac{\min \{ |k|^2 , c \}
}{2\sigma^2}   \right) 
\end{equation}
\end{lemma}

\begin{proof}
One can write the exponent in the right hand side of (\ref{boundexpTK}) as

\begin{eqnarray}
 \frac{i}{(2\pi)^{d/2}}
 \int_{\bbbr^d}\!\! \langle x, k\rangle   a(x) dx
-\frac{1}{(2\pi)^{d/2}} \int_0^1 (1-s)
\int_{\bbbr^d}\!\! \langle k,x\rangle^2  e^{is \langle
k,x\rangle } a_\varepsilon(x) dx ds. \label{Eq6.1bb}
\end{eqnarray}
Define a continuous function with values in Hermitian matrices by
\begin{equation}
a_{i,j}^{(2)}(k) :=2 \int_0^1 (1-s)
\int_{\bbbr^d} x_i x_j  \cos( s \langle
k,x \rangle )  a_\varepsilon(x) dx ds
\end{equation}
with $a_{i,j}^{(2)}(0) = a_{i,j}^{(2)}$. The eigenvalues of $(a_{i,j}^{(2)}(k))_{i,j=1, \ldots , d}$ are continuous functions as well, see Chapter~I.3 in \cite{Re69}. Hence for any $\sigma$ there exists a $K$ such that the smallest eigenvalue for all $|k| \leq K$ is still larger then $\sigma$. 

As $a$ is integrable we have that $ \hat a$ is continuous with $\lim_{k \rightarrow \infty} \hat a (-k) =0$ which yields $ \inf_{|k| \geq K} \mathrm{Re}(\hat a (-k)-\hat a(0))) < 0$ and hence the bound follows.
\end{proof}

Let us use this lemma and the bound  
\begin{equation}
\left| (\sqrt{2\pi}L)^{-d} \widehat{ \Ii_{[-L/2,L/2]^d}}*\hat{z}(j/L) \right| \leq 
\frac{1}{L^d} \int_{[-L/2,L/2]^d} z(x) dx \leq \| z \|_\infty
\end{equation}
in order to estimate 
\begin{eqnarray*}
&& (\sqrt{2\pi}L)^{-d} \sum_{j \in \bbbz^d\, : \, j \neq 0}  e^{t(2\pi)^{d/2}\mathrm{Re}(\hat a (-j/L)-\hat a(0))}
| \hat\varphi(-j/L)| \widehat{ \Ii_{[-L/2,L/2]^d}}*\hat{z}(j/L)
\\ && \leq \| z \|_\infty \sum_{j \in \bbbz^d\, : \, j \neq 0}  e^{t(2\pi)^{d/2}\mathrm{Re}(\hat a (-j/L)-\hat a(0))}
| \hat\varphi(-j/L)| \\
&& C \| z \|_\infty  L^d e^{-t \frac{\min\{ L^{-2} ,c\} }{2 \sigma}^2} \rightarrow 0,
\end{eqnarray*}  
for some constant $C$. The latter converge to zero so long $L^2 \ln L /t  \rightarrow 0$.

One may think that one can use Lemma~\ref{lemboundeaTK} to show that the time asymptotic of the semi-group $e^{tA}\varphi$ is the same as the one of the heat semi-group also in the infinite particle system, that is $L = \infty$. However, even when one ignores that Lemma~\ref{lemboundeaTK} gives different exponential rates, one sees that the difference between the kernels can at best be bounded near $k=0$ by a term of the form $|k|t^M e^{-ct|k|^2}$, where $M$ is the order of the derivative considered. In the limit $t \rightarrow \infty$ this can be uniformly bounded by a term $|k|^{-2M+1}$, which becomes worth when one uses higher derivatives.

\subsubsection{Hydrodynamic limit:}
In finite volume also the study of the hydrodynamic limit is rather easy. In Fourier coordinates the exponent in (\ref{Natal2}) takes the form
\begin{equation}\label{eqhydrofv}
(\sqrt{2 \pi}  L)^{-d} \sum_{j \in \bbbz^d} e^{t (2\pi)^{d/2} \varepsilon^{-2} (\hat{a}_\varepsilon(-\varepsilon j / L)-\hat{a}(0))} \hat\varphi (-j/ L) \widehat{ \Ii_{[-L/2,L/2]^d}}*\hat{z}(j/ L)
\end{equation}
The only $\varepsilon$ dependence is in the jump kernel.
As $a$ is integrable, $\hat{a}_\varepsilon(-\varepsilon x)-\hat{a}(0))$ converges locally uniformly and without loss of generality we can assume that $\hat{\varphi}$ has compact support, the hydrodynamic limit is established immediately without any further assumptions on $z$. 
Again because of the spectral gap we only need that $a$ has second moment and no further assumptions on $z$ beside boundedness in contrast to Proposition~\ref{PropHyd} and \ref{PropHydW}.

In Proposition~\ref{Prohydrogen} we were able to establish the hydrodynamic limit for a general $z$, but only under the assumption that $a$ has all higher moments finite with increasing requirements for larger dimension $d$ of the one-particle system. 

However, one can establish the hydrodynamic limit with only little stronger requirements on $a$, namely that the third moment is finite for example (any power larger than $2$ should be sufficient as well).  We have to improve the bound a bit in the following way. As all derivatives of $t (2\pi)^{d/2} \varepsilon^{-2} (\hat{a}_\varepsilon(-\varepsilon k)-\hat{a}(0))$ in $k$ are polynomially bounded in $k$ and uniformly bounded in $\varepsilon$, one can use the mean value theorem for the two exponentials in (\ref{eq6.8}) to bound
\begin{equation} \label{eq6.8fv}
\left| 
e^{\varepsilon^{-\kappa}t(2\pi)^{d/2}(\hat{a}_\varepsilon(-\varepsilon
k)-\hat{a}(0))} - e^{ it \langle
a^{(1)},k \rangle  -\frac{0^{2-\kappa}}{2} t \langle
a^{(2)}k,k\rangle }\right|  \leq   C(1+|k|^2)^{1/2} \varepsilon .
\end{equation}
Using that the term $\widehat{ \Ii_{[-L/2,L/2]^d}}*\hat{z}(j/\varepsilon L)$ can be bounded by $\| z\|_\infty L^d$ and
assuming that $ \varepsilon L^d  \rightarrow 0$ we can replace the exponent in (\ref{eqhydrofv}) by $ it \langle
a^{(1)},k \rangle  -\frac{0^{2-\kappa}}{2} t \langle
a^{(2)}k,k\rangle$. Taking the inverse Fourier transform we obtain
\begin{equation}
\int_{\bbbr^d} e^{-t \langle a^{(1)} ,\nabla \rangle - t \langle \nabla, a^{(2)} \nabla \rangle /2}\varphi(x) P_Lz(x) dx
\end{equation}
which converges for example for bounded $z$ and $\varphi \in L^1$ to the same expression with $P_L z$ replaced by $z$.

In both cases, long time asymptotic and hydrodynamic limit, the spectral gap coming from the finite volume is simplifying the situation a lot. The spectral gap being of order $L^{-2}$ is too small to describe the physical relevant decay of the system. However, the infinite volume variant of the two aforementioned cases show that even for $L = \infty$ one has the required result, but which must be based on a different mechanism.

\section{Non-equilibrium dynamics}\label{nonled}

In this section we widen the class of initial distributions to
measures far from equilibrium, that is we consider all probability measures $\mu$ on $\Theta$ as initial distributions subject only to a mild mixing condition. This means that we
consider the processes constructed as in Section \ref{Section2} but
not necessarily with a Poissonian initial distribution.
Assuming enough mixing of the initial measure $\mu$, namely (\ref{CondUrs}), we are able to generalize, incorporating ideas from \cite{DSS82}, Proposition~\ref{PrAsya}, see Proposition~\ref{PrAsyGibbs} in Subsection~\ref{SubsecGenTemp} and Proposition~\ref{Prohydrogen}, see Theorem~\ref{ThHydGibbs} in Subsection~\ref{SubsecGenHydro}.

We formulate the mixinig requirement in terms of the second Ursell function (factorial cumulant), which can be expressed in terms of the first and second correlation function (factorial moments) defined in Subsection~\ref{Subsection2.2},
namely
\[
u_\mu^{(2)}(x,y) := k_\mu(\{x,y\}) - k_\mu(\{x\}) k_\mu(\{y\}).
\]

We denote in the following the first correlation function $x \mapsto k_\mu(\{x\})$ by $\rho_\mu$.
The condition on the second Ursell function, (\ref{CondUrs}), is a rather weak mixing or decay of correlation condition.  In Subsection~\ref{SecGibbs} we show that this condition is
 fulfilled, in particular, by Gibbs measures in the high temperature
regime. It will also hold beyond that regime even in the presence of a phase transition,
 cf.~e.g.~\cite{PrBook}.

\subsection{Long time asymptotic \label{SubsecGenTemp}}

Recalling (\ref{onebog}), the Laplace transform of the one-dimensional distribution
$P^{\mathbf{X}}_{\mu,t}$ can be expressed in terms of a one-particle
system, i.e., for all non-negative $f\in \mathcal{S}(\bbbr^d)$
\begin{equation}\label{eq6.12}
\int_\Gamma e^{-\langle f , \gamma \rangle }
P^{\mathbf{X}}_{\mu,t}(d\gamma) = \int_\Theta
e_B(e^{tA}(e^{-f}-1),\gamma) \mu(d\gamma) = \int_\Theta
e^{\langle \ln (e^{tA}(e^{-f}-1)+1),\gamma\rangle} \mu(d\gamma).
\end{equation}

\begin{proposition}
\label{PrAsyGibbs} Let $\mu$ be a measure on $\Theta$ which has first and second correlation function. Assume that $\mu$ fulfills the following mixing condition
\begin{equation}\label{CondUrs}
\sup_{x \in \bbbr^d} \int_{\bbbr^d}u^{(2)}_\mu(x, y) dy < \infty.
\end{equation}
In addition, we assume
that the Fourier transform of the first correlation function
$\rho_\mu$ is a signed measure.
Then, the one-dimensional distribution $P^{\mathbf{X}}_{\mu,t}$
converges weakly to $\pi_{\mathrm{mean}(\rho_\mu)}$ when $t$
tends to infinity.
\end{proposition}

\begin{proof}
Given a non-negative $f \in \mathcal{S}(\bbbr^d)$ such that $-1 \leq e^{-f}-1\leq  0$,
let  $\varphi:=1-e^{-f}$.
Using (\ref{eq6.12}) and $|e^{-x} - e^{-y}|\leq |x-y|$ for $x,y \geq 0$ one obtains that
\begin{eqnarray*}
&&\left| \int_\Theta
e^{\langle \ln (e^{tA}(e^{-f}-1)+1),\gamma\rangle} \mu(d\gamma) - e^{\mathrm{mean}(\rho_\mu) \int_{\bbbr^d}(e^{-f}-1)(x)  dx}\right|
\\&\leq&  \int_\Theta \left|
\langle \ln (-e^{tA}\varphi+1),\gamma\rangle - \mathrm{mean}(\rho_\mu) \int_{ \bbbr^d}\varphi(x)  dx\right| \mu(d\gamma)
\end{eqnarray*}
Using that  $0\leq x -\ln(x+1)\leq x^2 $ for $-1/2\leq x \leq 0$ and that
 $\|e^{tA}\varphi\|_u$ tends to zero for $t \rightarrow
\infty$, because $\hat{\varphi} \in L^1(\bbbr^d,dx)$ and
\[
|e^{tA}\varphi(x)|  \leq \frac{1}{(2\pi)^{d/2}}\int_{\bbbr^d}
dk e^{t(2\pi)^{d/2}\mathrm{Re}(\hat{a}(k)-\hat{a}(0))}
|\hat{\varphi}(k)|
\]
and
$(e^{tA})_{t\geq 0}$ is an $L^1(\bbbr^d,dx)$-contraction we get that
\begin{eqnarray*}
&&\int_\Theta \left|
\langle \ln (e^{tA}\varphi+1),\gamma\rangle - \langle e^{tA}\varphi,\gamma\rangle \right| \mu(d\gamma)
\leq \int_{\bbbr^d}  \left(e^{tA}\varphi(x) \right)^2 \rho_\mu(x) dx \\ &\leq& \|e^{tA}\varphi\|_u \int_{\bbbr^d}  e^{tA}\varphi(x) \rho_\mu(x)dx \leq  \|e^{tA}\varphi\|_u \|\varphi\|_{L^1} \|\rho_\mu\|_u.
\end{eqnarray*}
This means it is sufficient to estimate
\begin{eqnarray*}
\lefteqn{\int_\Theta \left|
\langle e^{tA}\varphi,\gamma\rangle - \mathrm{mean}(\rho_\mu) \int_{ \bbbr^d}\varphi(x)  dx\right| \mu(d\gamma)
}\\&\leq& \int_\Theta \left|
\langle e^{tA}\varphi,\gamma\rangle -  \int_{ \bbbr^d}e^{tA}\varphi(x) \rho_\mu(x) dx\right| \mu(d\gamma)
\\&&+ \left|  \int_{ \bbbr^d}e^{tA}\varphi(x) \rho_\mu(x) dx - \mathrm{mean}(\rho_\mu) \int_{ \bbbr^d}\varphi(x)  dx\right|.
\end{eqnarray*}
The last term converges to zero because of Lemma~\ref{LemasyA} and
Corollary~\ref{Colmeana}. Due to the decay properties of the covariance function (\ref{CondUrs})
\begin{eqnarray*}
 \lefteqn{\int_\Theta \left|
\langle e^{tA}\varphi,\gamma\rangle -  \int_{ \bbbr^d}e^{tA}\varphi(x) \rho_\mu(x) dx\right| \mu(d\gamma)}\\
&\leq& \int_{\bbbr^d} \int_{\bbbr^d}e^{tA}\varphi(x) e^{tA}\varphi(y) u^{(2)}_\mu(x,y)dx dy
+ \int_{\bbbr^d} \left( e^{tA}\varphi(x)\right)^2 \rho_\mu(x) dx\\&\leq&  \|e^{tA}\varphi\|_u \|e^{tA}\varphi\|_{L^1} \left( \sup_{x \in \bbbr^d}  \int_{\bbbr^d}u^{(2)}_\mu(x,y) dy + \| \rho_\mu\|_u \right),
\end{eqnarray*}
which implies the result, as $e^{tA}$ is an $L^{1}(\bbbr^d,dx)$ contraction and $\|e^{tA}\varphi\|_u$
converges to zero.
\end{proof}

%
%

\subsection{Hydrodynamic limits \label{SubsecGenHydro}}

As in Subsection~\ref{SubSecHydro} we want to study the rescaled
empirical field $n_t^{(\varepsilon)}(\varphi,\mathbf{X}) =
\varepsilon^d\left\langle\varphi(\varepsilon\cdot),
\mathbf{X}_{\varepsilon^{-\kappa}t}\right\rangle$. Since we do not
have any longer a natural parameter associated to the initial measure, one
cannot formulate something like slowly varying intensities. However,
one sees that a possible framework is to work with a quite arbitrary
sequence of initial measures $(\mu_\varepsilon)_{\varepsilon >0}$.
The main restriction on this sequence is that one has to assume a
particular convergence for the first correlation measure described below in more details
and the mixing condtion (\ref{CondUrs}) uniformly in $\varepsilon$. In Corollary~\ref{Lem1Ursell} we prove
that these conditions are fulfilled by Gibbs measures in the
high temperature regime for slowly varying intensity $z(\varepsilon
\cdot)$. Furthermore, the limit is identified.

\begin{theorem}
\label{ThHydGibbs} Assume that $a$ has all moments finite. Let
$(\mu_{\varepsilon})_{\varepsilon > 0}$ be a sequence of measures on
$\Theta$ such that
$\rho_{\mu_\varepsilon}$ is uniformly bounded in $\varepsilon$,
the limit $\lim_{\varepsilon \rightarrow 0^+}
\rho_{\mu_\varepsilon}(\{x/\varepsilon\}) =: \rho_0(x)$ exists for all
$x \in \bbbr^d$ and the following mixing condition
\[
\sup_{x \in \bbbr^d,\varepsilon >0} \int_{\bbbr^d} u^{(2)}_{\mu_{\varepsilon}}(x,y) dy <\infty
\]
holds.
Then, for each $t\geq 0$, the following limit exists for all non-negative
$\varphi\in\mathcal{D}(\bbbr^d)$
\begin{equation}
\label{eqhydrolim} \lim_{\varepsilon\to
0^+}\int_\Gamma\mu_\varepsilon(d\gamma)\,
\mathbf{E}_\gamma\left[e^{-n^{(\varepsilon)}_t(\varphi,\mathbf{X})}\right]
=: \int_{\mathcal{D}'(\bbbr^d)} \delta_{\rho_t}(d\omega)
e^{-\langle \varphi, \omega \rangle},
\end{equation}
and the conclusions of Subsection~\ref{SubSecHydro} continue to hold.
\end{theorem}

\begin{proof}
Given a non-negative $\varphi \in \mathcal{S}(\bbbr^d)$ according to (\ref{eq6.12})
one may write (\ref{eqhydrolim}) as
\begin{equation}
\left| \int_\Theta\mu_\varepsilon(d\gamma)\,
e^{\left\langle \ln \left(e^{\varepsilon^{-\kappa}tA_\varepsilon}
(e^{-\varepsilon^{d}\varphi(\varepsilon \cdot)} -1)+1\right),\gamma\right\rangle}
- e^{-\int_{\bbbr^d} \varphi(x) \rho_t(x) dx}\right|
\end{equation}
Using $|e^{-x} - e^{-y}|\leq |x-y|$ for $x,y \geq 0$ one can bound this by
\begin{eqnarray}\label{eqesthydro1a}
  \int_\Theta \left|
\langle \ln \left(e^{\varepsilon^{-\kappa}tA_\varepsilon}
(e^{-\varepsilon^{d}\varphi(\varepsilon \cdot)} -1)+1\right),\gamma\rangle +  \int_{ -\bbbr^d}\varphi(x) \rho_t(x) dx\right| \mu_\varepsilon(d\gamma).
\end{eqnarray}
Proceeding
as in the proof of Proposition~\ref{PrAsyGibbs} we get that (\ref{eqesthydro1a}) can be bounded by
\begin{eqnarray*}
&\leq& \left\| e^{\varepsilon^{-\kappa}tA_\varepsilon}
(e^{-\varepsilon^{d}\varphi(\varepsilon \cdot) }-1)\right\|_u  \left\|e^{\varepsilon^{-\kappa}tA_\varepsilon}
(e^{-\varepsilon^{d}\varphi(\varepsilon \cdot) }-1)\right\|_{L^1}\|\rho_{\mu_\varepsilon}\|_u \\
&&+ \left\|e^{\varepsilon^{-\kappa}tA_\varepsilon}
\left(e^{-\varepsilon^{d}\varphi(\varepsilon \cdot)} -1 +
\varepsilon^{d}\varphi(\varepsilon \cdot)\right) \right\|_{L^1} \|\rho_{\mu_\varepsilon}\|_u .
\end{eqnarray*}
The latter expression is at least of order $\varepsilon^d$ according to Lemma~\ref{lemmHyd} in the Appendix.

Thus it remains to bound 
\begin{eqnarray}
 && \int_\Theta \left|\langle e^{\varepsilon^{-\kappa}tA_\varepsilon}
\varepsilon^{d}\varphi(\varepsilon \cdot),\gamma\rangle
 -  \int_{\bbbr^d}\varphi(x) \rho_t(x) dx\right| \mu_\varepsilon(d\gamma)\\
& \leq &
  \int_\Theta \left|\langle e^{\varepsilon^{-\kappa}tA_\varepsilon} 
\varepsilon^{d}\varphi(\varepsilon \cdot),\gamma\rangle
 -  \int_{\bbbr^d}  \left(e^{\varepsilon^{-\kappa}tA_\varepsilon}\varphi(\varepsilon \cdot) \right)(x/\varepsilon)
\rho_{\mu_\varepsilon}(x/\varepsilon) dx\right| \mu_\varepsilon(d\gamma)
\nonumber \\ &&+ \left|\int_{\bbbr^d}  \left(e^{\varepsilon^{-\kappa}tA_\varepsilon}\varphi(\varepsilon \cdot) \right)(x/\varepsilon)
\rho_{\mu_\varepsilon}(x/\varepsilon) dx -  \int_{\bbbr^d}\varphi(x) \rho_t(x) dx \right| .\label{eqsecsum}
\end{eqnarray}
In Proposition~\ref{Prohydrogen} we show that
$e^{\varepsilon^{-\kappa}tA_{\varepsilon}}\varphi(\varepsilon
\cdot)(x/\varepsilon)$ converges in $\mathcal{S}(\bbbr^d)$ to
\[
e^{-t\langle a^{(1)},\nabla \rangle + t 0^{2-\kappa}/2 \langle
\nabla, a^{(2)} \nabla \rangle}\varphi(x).
\]
By assumption, $\rho_{\mu_\varepsilon}(\{x/\varepsilon\})$  converges, in particular, in
$\|\cdot\|_{0,-d-1,2}$, cf.~(\ref{defnorm}), and thus one obtains
\begin{eqnarray*}
&&\lim_{\varepsilon \rightarrow 0^+} \int_{\bbbr^d
}e^{\varepsilon^{-\kappa}tA_{\varepsilon}}
\varphi(\varepsilon \cdot)(x/\varepsilon)
\rho_{\mu_\varepsilon}(\{x/\varepsilon\})  dx \\
&=& \int_{\bbbr^d }e^{t\left(-\langle a^{(1)},\nabla \rangle +
0^{2-\kappa}/2 \langle \nabla, a^{(2)} \nabla
\rangle\right)}\varphi(x) \rho_0(x) dx.
\end{eqnarray*}
Hence the second summand in (\ref{eqsecsum}) converges to zero. The first summand can be bounded by the second Ursell functions as in the proof of Proposition~\ref{PrAsyGibbs}
\begin{eqnarray*}
&&\leq \left\|e^{\varepsilon^{-\kappa}tA_\varepsilon}
\varepsilon^{d}\varphi(\varepsilon \cdot) \right\|_u \left\|e^{\varepsilon^{-\kappa}tA_\varepsilon}
\varepsilon^{d}\varphi(\varepsilon \cdot) \right\|_{L^1} \left( \sup_{x \in \bbbr^d,\varepsilon>0} \int_{\bbbr^d} u^{(2)}_{\mu_{\varepsilon}}(x,y) dy + \|\rho_{\mu_\varepsilon} \|_u \right).
\end{eqnarray*}

\end{proof}

\subsection{Application to Gibbs measures \label{SecGibbs}}

In this subsection we prove that the hypothesis for the results of the previous subsection are fulfilled
for a concrete class of non-equilibrium measures, namely, for Gibbs measures in the high
temperature low activity regime.

In order to recall the definition of a Gibbs measure, first we have
to introduce a pair potential
$V:\bbbr^d\to\bbbr\cup\{+\infty\}$, that is, a measurable
function such that $V(-x)=V(x)\in\bbbr$ for all
$x\in\bbbr^d\setminus\{0\}$. For $\gamma\in\Gamma$ and
$x\in\bbbr^d\setminus\gamma$ we define a relative energy of
interaction between a particle located at $x$ and the configuration
$\gamma$ by
\begin{eqnarray*}
E(x,\gamma ):=\left\{
\begin{array}{cl}
\displaystyle\sum_{y\in \gamma }V(x-y), & \mathrm{if\;}
\displaystyle\sum_{y\in \gamma }|V (x-y)|<\infty \\
&  \\
+\infty , & \mathrm{otherwise}
\end{array}
\right. .
\end{eqnarray*}
A probability measure $\mu$ on $\Gamma$ is called a Gibbs measure
corresponding to $V$, an intensity function $z\geq 0$, and an
inverse of temperature $\beta$ whenever it fulfills the
Georgii-Nguyen-Zessin equation \cite[Theorem 2]{NZ79}
\begin{equation}
\int_\Gamma \mu(d\gamma)\,\sum_{x\in \gamma }H(x,\gamma )=
\int_\Gamma \mu(d\gamma)\int_{\bbbr^d}dx\,z(x) H(x,\gamma \cup
\{x\}) e^{-\beta E(x,\gamma )}\label{1.3}
\end{equation}
for all positive measurable functions
$H:\bbbr^d\times\Gamma\to\bbbr$. This definition is
equivalent to the definition via DLR-equation, see
\cite{Ge76,NZ79,P81,K00}. We observe that for $V\equiv 0$
(\ref{1.3}) reduces to the Mecke identity, which yields an
equivalent definition of the Poisson measure $\pi_z$ \cite[Theorem
3.1]{Me67}. We also note that for either $V\equiv 0$ and $z$ not
being a constant or $V\not=0$, a Gibbs measure neither is a
reversible nor an invariant initial distribution for the  free
Kawasaki dynamics under consideration. In order to have
thermodynamical behavior we assume that $V$ is stable, i.e., there
exists a $B>0$ such that $\sum_{\{x,y\} \subset \eta} V(x-y) \geq -
B |\eta|$ for all configurations $\eta \in \Gamma_0$. Furthermore,
we shall assume that the parameters $\beta,z$ are small (high
temperature low activity regime), i.e.,
\[
\|z\|_u e^{2\beta B+1} C(\beta) < 1,
\]
where $C(\beta):=\int_{\bbbr^d}dx\, e^{-\beta |V(x)|}-1$. These
conditions are, in particular, sufficient to insure the existence of
Gibbs measures, cf.~\cite{Ru69} and \cite{PrYu17}. Moreover, the correlation functions
corresponding to such measures exist and fulfil a Ruelle bound defined in
Section \ref{Section2}, and thus, as noted there, they are supported
on $\Theta$.

In the high temperature low activity regime one has rather detailed
information about the Ursell functions (factorial cumulants) $u_\mu:\Gamma_0\rightarrow \bbbr$
corresponding to $\mu$, which are the bounded measurable functions
such that for all $f \in
\mathcal{D}(\bbbr^d)$ holds
\begin{equation}
\label{DefUrs} \int_{\Gamma} e^{\langle f,\gamma \rangle}
\mu(d\gamma) = \exp \left( \int_{\Gamma_0} \prod_{y \in
\eta}(e^{f(y)}-1) u_\mu(\eta ) \lambda(d\eta) \right).
\end{equation}	
The function $x \rightarrow u_\mu(\{x\})$ coincides with the first
correlation function of $\mu$.

%
We present the results necessary for the following. For further
details, see e.g. \cite{DuSoIa75}, \cite{MaMi91} and see also
\cite{Ku98}.

The Ursell functions can be expressed in terms of a sum over all
connected graphs weighted by the Meyer-functions
\begin{equation}
k(\xi) := \sum_{G \in \mathcal{G}_c(\xi)} \prod_{\{x,y\} \in G}
(e^{-\beta V(x-y)}-1)
\end{equation}
where $\mathcal{G}_c(\xi)$ denotes the set of all connected graphs
with vertex set $\xi$:
\begin{equation}
u_\mu(\eta) := \int_{\Gamma_0}\lambda_z(d\xi)  k(\eta \cup \xi)
\prod_{x \in \eta} z(x).
\end{equation}
Actually, the following bound is the key result of the cluster
expansion of Penrose-Ruelle type
\begin{equation}
\label{boundTree} | k(\xi) | \leq e^{2\beta B |\xi|} \sum_{T \in
\mathcal{T}(\xi)} \prod_{\{x,y\} \in G} (e^{-\beta |V(x-y)|}-1),
\end{equation}
where $\mathcal{T}(\xi)$ denotes the set of all trees with set of
vertices $\xi.$ This leads to the following integrability bound
\begin{eqnarray}
\label{boundUrsell1} \lefteqn{\int_{\bbbr^{dn}}
|u_\mu(\{x,y_1,\ldots,y_n\})|
z(y_1) dy_1\ldots z(y_n)dy_n }\nonumber \\
 &\leq& e^{(2\beta B+1) (n+1)}
\left(\|z\|_u C(\beta)\right)^{n} \sum_{m=0}^\infty
\frac{(n+m+1)!}{m!} \left(e^{2\beta B +1}\|z\|_u C(\beta) \right)^m.
\end{eqnarray}
In particular the mixing condition (\ref{CondUrs}) of Proposition~\ref{PrAsyGibbs} and Theorem~\ref{ThHydGibbs} holds.

We show that Gibbs
measures in the high temperature regime with a translation invariant
potential fulfill the assumptions of Proposition~\ref{PrAsyGibbs}.

\begin{corollary}
Let $z\geq 0$ be a bounded measurable function which Fourier
transform is a bounded signed measure. Let $\mu$ be a Gibbs measure
corresponding to a translation invariant potential $V$ described above,
inverse temperature $\beta$ and activity $z$ which are in the high temperature low
activity regime. Then the first correlation function
$\rho_\mu$ has as Fourier transform a measure and the arithmetic mean
\[
\mathrm{mean}(\rho_\mu) =
\frac{1}{(2\pi)^{d/2}}\sum_{n=0}^\infty\frac{1}{n!}
\int_{\bbbr^{dn}} \overline{\hat{k_r}(p_1,\ldots,p_n)}
\hat{z}(\{p_1+\ldots+p_n\})\hat{z}(dp_1)\cdot\ldots\cdot
\hat{z}(dp_n),
\]
where
\[ k_r(y_1,\ldots,y_n):=\sum_{G \in {\tilde \mathcal{G}}_c}
\prod_{\{x_1,x_2\} \in G} (e^{-\beta V(x_1-x_2)}-1)
\]
and where ${\tilde\mathcal{G}}_c$ denotes the set of all connected graphs
with vertex set $(0,y_1,\ldots,y_n)$. As a consequence, all assumptions
 of Proposition~\ref{PrAsyGibbs} are fulfilled.
\end{corollary}

\begin{proof}
Due to the translation invariance of $V$ and cluster expansion one can rewrite the first
correlation function as
\[
\rho_\mu(x) = \sum_{n=0}^\infty
\frac{1}{n!}\int_{\bbbr^{dn}} \sum_{G \in
{\tilde\mathcal{G}}_c} \prod_{\{y_1,y_2\} \in G} (e^{-\beta
V(y_1-y_2)}-1) z(x_1-x)\cdot \ldots \cdot z(x_n-x) z(x) dx_1\cdot
\ldots \cdot dx_n .
\]
By (\ref{boundTree}) and (\ref{boundUrsell1}) the function $k_r$ is
integrable, and thus  $\hat{k}_r$ is a continuous function which decays
to zero at infinity. Hence
\begin{eqnarray*}
k_r(x_1-x,\ldots,x_n-x) &=& (2\pi)^{-nd/2} \int_{\bbbr^{dn}}
dp_1\ldots dp_n dp e^{ip_1 x_1}\cdot \ldots \cdot e^{ip_nx_n}e^{ipx}
\\
&&\cdot \hat{k}_r(p_1,\ldots,p_n) \delta(p-p_1-\ldots-p_n) .
\end{eqnarray*}
To compute the Fourier transform of $\rho_\mu$ in the weak
sense it remains to compute the Fourier transform of
\begin{eqnarray*}
\lefteqn{\varphi(x) z(x_1-x)\cdot \ldots \cdot z(x_n-x) z(x)}\\
&=&(2\pi)^{-(n+2)d/2} \int_{\bbbr^{dn}} \hat{z}(dp_1)
e^{ip_1x_1}\cdot\ldots\cdot \hat{z}(dp_n) e^{ip_nx_n}\ dp\ e^{ipx}\\
&&e^{-i(p_1+\ldots+p_n)x} \left(\hat{\varphi}\ast\hat{z}\right)(p),
\end{eqnarray*}
for any $\varphi\in \mathcal{D}(\bbbr^d)$. Summarizing, we obtain that
\begin{eqnarray*}
\int_{\bbbr^d} dx \varphi(x) \rho_\mu(x) &=&
\sum_{n=0}^\infty\frac{1}{n!} \int_{\bbbr^{dn}}
\hat{z}(dp_1)\cdot\ldots\cdot \hat{z}(dp_n)\hat{z}(dp) \\ &&\cdot
\overline{\hat{k_r}(p_1,\ldots,p_n)}\hat{\varphi}(p_1+\ldots+p_n-p).
\end{eqnarray*}
\end{proof}

As in Subsection~\ref{SubSecHydro}, we consider initial measures with
a slowly varying intensity, i.e., Gibbs measures corresponding to
$\beta$, $V$, and $z(\varepsilon \cdot)$, which we denote by
$\mu_{\varepsilon}$. As $\|z\|_u$ is unchanged, all scaled measures
$\mu_\varepsilon$ remain in the high temperature low activity
regime and the bound (\ref{boundUrsell1})
holds
uniformly in $\varepsilon >0$. For Gibbs measures which are not
Poisson measures, the first correlation function is not any longer
just the intensity. The function appearing as initial value in the
limiting partial differential equation is the scaling limit of the
first correlation function and not just the unscaled activity. Let
us describe what is the scaling limit of the first correlation
function. Denote by $\rho^{\mathrm{equi}}_c$ the correlation function
corresponding to the Gibbs measure with constant activity $c$,
temperature $\beta$ and potential $V$. Due to the translation
invariance of $V$ this correlation function is a constant function.
Given a function $z\geq 0$ denote by $x \mapsto \rho^{\mathrm{equi}}_{z(x)}$
the function which associates to each $x$ the constant value of the
first correlation function of the Gibbs measure with constant
activity $z(x)$. This function is the scaling limit of the first
correlation of $\mu_\varepsilon$. Note that $\rho_\mu (x)=
u_\mu(\{x\})$.

\begin{corollary}
\label{Lem1Ursell} Given a bounded measurable function $z\geq 0$, a
potential $V$, and an inverse  temperature $\beta$  fulfilling the
conditions of the high temperature and low activity regime, the
corresponding Gibbs measures fulfill all assumptions of
Theorem~\ref{ThHydGibbs}. Moreover,
\begin{equation}
\rho_0(x) := \lim_{\varepsilon \rightarrow 0^+}
u_{\mu_\varepsilon}(\{x/\varepsilon\}) = \rho^{\mathrm{equi}}_{z(x)}.
\end{equation}
\end{corollary}

\begin{proof}
According to the cluster expansion of $u_\mu$ and the translation
invariance of $V$ we obtain that
\begin{equation}
u_{\mu_\varepsilon}(\{x/\varepsilon\}) = \sum_{n=1}^\infty
\frac{1}{n!}\int_{\bbbr^{dn}} dy_l \sum_{G \in
\mathcal{G}_c(\{0,\ldots ,n\})} \prod_{\{i,j\} \in G} (e^{-\beta
V(y_{i}-y_{j})}-1) \prod_{l=1}^n z(\varepsilon y_l +x ) ,
\end{equation}
where $y_0 := 0$. Due to (\ref{boundUrsell1}) the above expression
is uniformly integrable in $\varepsilon$ and in $x$, yielding the
claimed result.
\end{proof}

\appendix

\section{Estimates for hydrodynamic limits}
\label{appest}

\begin{lemma}
\label{lemmHyd} Let $\varphi\in\mathcal{S}(\bbbr^d)$ and $0\leq
\varepsilon \leq 1$ be given. Then
\begin{eqnarray}
\|e^{tA}(e^{\varepsilon^d \varphi(\varepsilon \cdot)}-1)\|_u
&\leq& \varepsilon^d \|\varphi\|_u e^{\|\varphi\|_u} \nonumber \\
\|e^{tA}(e^{\varepsilon^d \varphi(\varepsilon
\cdot)}-1)\|_{L^1(\bbbr^d,dx)}
&\leq& \|\varphi \|_{L^1(\bbbr^d,dx)} e^{\|\varphi \|_{L^1(\bbbr^d,dx)}}\label{Natal9} \\
\|e^{tA}(e^{\varepsilon^d \varphi(\varepsilon \cdot)}-1 -
\varepsilon^d \varphi(\varepsilon \cdot))\|_{L^1(\bbbr^d,dx)}
&\leq& \varepsilon^d \|\varphi \|^2_{L^1(\bbbr^d,dx)}
e^{\|\varphi \|_{L^1(\bbbr^d,dx)}}\nonumber
\end{eqnarray}
\end{lemma}

\begin{proof}
On the one hand $(e^{tA})_{t\geq 0}$ is a contraction semigroup with
respect to the the supremum norm. So
\[
\|e^{tA}(e^{\varepsilon^d \varphi(\varepsilon \cdot)}-1)\|_u \leq
\|e^{\varepsilon^d \varphi(\varepsilon \cdot)}-1\|_u \leq
\sum_{n=1}^\infty \frac{\varepsilon^{nd}}{n!} \|\varphi\|_u^n.
\]
On the other hand $(e^{tA})_{t\geq 0}$ is a contraction semigroup on
$L^1(\bbbr^d,dx)$. Therefore, the left-hand side of
(\ref{Natal9}) is bounded by
\[
\int_{\bbbr^d}dx\, \frac{|e^{\varepsilon^d \varphi(x)}-1|}
{\varepsilon^d} \leq  \|\varphi\|_{L^1(\bbbr^d,dx)} \sum_{n=1}^\infty
\frac{\varepsilon^{d(n-1)}}{n!} \|\varphi\|_{L^1(\bbbr^d,dx)}^{n-1}\leq
\|\varphi \|_{L^1(\bbbr^d,dx)} e^{\|\varphi \|_{L^1(\bbbr^d,dx)}}.
\]
The last inequality follows by a similar computation.
\end{proof}

%
%
%
%
Let us introduce the following two equivalent systems of norms for
the locally convex topological vector space
$\mathcal{S}(\bbbr^d)$ for $A \in \bbbn$ and $M\geq 0$:
\begin{eqnarray}
\|f \|_{A,M,u} &:=& \sum_{\stackrel{ \scriptstyle
\alpha \in \bbbn_0^d}{|\alpha|\leq A}}
\sup_{x \in \bbbr^d} \left| D^\alpha f (x) \right|
\left(1 + |x|^2\right)^M \\
\|f \|_{A,M,2} &:=& \sum_{\stackrel{ \scriptstyle \alpha \in
\bbbn_0^d}{|\alpha|\leq A}} \left(\int_{ \bbbr^d} \left|
D^\alpha f (x) \right|^2 \left(1 + |x|^2\right)^M dx\right)^{1/2}
\label{defnorm}
\end{eqnarray}

\begin{lemma}
Let $f_1,f_2 \in C^\infty(\bbbr^d;\bbbc)$ be two $C^\infty$-functions  with
non-positive real part such that for an $A \in \bbbn$ and a
$M\geq 0$ one has $\|f_i\|_{A,-M,u}< \infty$. Then there exists a
constant $C$ depending on $A$ and on $\|f_i\|_{A,-M,u}$ such that
\begin{equation}
\|e^{f_1} -e^{f_2}\|_{A,-(A+1)M,2}
\leq C \| f_1 -f_2 \|_{A,-M,2}.
\end{equation}
\end{lemma}

\subsection*{Acknowledgment} We would like to thanks N.~Jacob, E.~Lytvynov, M.~R{\"o}ckner, W.~Hoh for fruitful discussions, T. Hilberdink for mentioning the theory slowly varying functions and T. Funaki for making us aware of \cite{DSS82}.  M.~J.~O.~was supported by the Portuguese national funds through FCT Funda{\c c}\~ao para a Ci{\^e}ncia e a Tecnologia, within the project UIDB/04561/2020. Financial support from DFG through the SFB 701 (Bielefeld University) and FCT through, POCI-2010, FEDER, PTDC, UIDB/MAT/04674/2020 are gratefully acknowledged.


\end{document}